\documentclass[11pt]{article}
\usepackage{a4wide}
\usepackage{amsmath,amsthm,amssymb}
\usepackage{units}
\usepackage{url}
\usepackage{calc, graphicx} 
\usepackage{bbm}

\newtheorem{thm}{Theorem}[section]
\newtheorem{lem}[thm]{Lemma}

{\theoremstyle{definition}
  
}

\newcommand{\N}{\ensuremath{\mathbb N}}
\newcommand{\C}{\ensuremath{\mathbb{C}}}
\newcommand{\R}{\ensuremath{\mathbb{R}}}
\newcommand{\Prob}{\ensuremath{\mathbb{P}}}
\newcommand{\bo}{\ensuremath{\mathrm{O}}}
\newcommand{\Var}{\ensuremath{\mathrm{Var}}}
\newcommand{\Cov}{\ensuremath{\mathrm{Cov}}}

\newcommand{\Nat}{\ensuremath{\mathbb{N}}}

\newcommand{\Erw}[1]{\ensuremath{\mathbb{E}\!\left[#1\right]}}
\newcommand{\E}{\ensuremath{\mathbb{E}}}

\newcommand{\coleq}{\ensuremath{\mathop{:}\!\!=}}

\newsavebox{\smlmat}
\savebox{\smlmat}{$\left[\begin{smallmatrix}1 & 2 \\ 2 &1\end{smallmatrix} \right]$ }

\title{P\'olya urns via the contraction method}
\author{Margarete Knape and Ralph Neininger\\
Institute for Mathematics\\
J.W.~Goethe University\\
60054 Frankfurt a.M.\\
Germany\\ \\
Email: {\tt \{knape,neiningr\}@math.uni-frankfurt.de}}
\begin{document}

\maketitle

\begin{abstract}
  We propose an approach to analyze the asymptotic behavior of P\'olya
  urns based on the contraction method. For this, a new combinatorial
  discrete time embedding of the evolution of the
  urn into random rooted trees is developed. A decomposition of these trees
  leads to a system of recursive distributional equations which
  capture the distributions of the numbers of balls of each color.
  Ideas from the contraction method are used to study such systems of
  recursive distributional equations asymptotically. We apply our
  approach to a  couple of  concrete P\'olya urns that  lead to limit
  laws with normal limit distributions, with non-normal limit
  distributions and with asymptotic periodic distributional behavior.
\end{abstract}

\noindent
{\textbf{MSC2010:} 60C05, 60F05, 60J05, 68Q25.


\section{Introduction}
In this paper, we develop an approach to prove limit theorems for P\'olya
urn models by the contraction method. We consider an urn with balls in
a finite number $m\ge 2$ of different colors, numbered by
$1,\ldots,m$. The evolution of a P\'olya urn is determined by an
$m\times m$ replacement matrix $R=(a_{ij})_{1\le i,j\le m}$ which is
given in advance together with an initial (time $0$) composition of
the urn with at least one ball. Time evolves in discrete steps. In
each step, one ball is drawn uniformly at random from the urn. If it
has color $i$ it is placed back into the urn together with $a_{ij}$
balls of color $j$ for all $j=1,\ldots,m$. The steps are iterated
independently. A classical problem is to identify the asymptotic
behavior of the numbers of balls of each color as the number $n$ of
steps tends to infinity. The literature on this problem, in particular
on limit theorems for the normalized numbers of balls of each color, is
vast. We refer to the monographs of Johnson and Kotz \cite{joko77} and
Mahmoud \cite{ma09} and the references and comments on the literature
in the papers of Janson \cite{Ja04}, Flajolet et al. \cite{flgape05}
and Pouyanne \cite{Pou08}.

A couple of approaches have been used to analyze the asymptotic
behavior of P\'olya urn models, most notably the method of moments,
discrete time martingale methods, embeddings into continuous time
multitype branching processes, and methods from analytic combinatorics
based on generating functions. All these methods use the ``forward''
dynamic of the urn process by exploiting that the distribution of the
composition at time $n$ given time $n-1$ is explicitly accessible.

In the present paper, we propose an approach based on a ``backward''
decomposition of the urn process. We construct a new embedding of the
evolution of the urn into an associated  combinatorial  random tree
structure growing in discrete time. Our associated tree can be
decomposed at its root (time $0$) such that the growth dynamics of the
subtrees of the  root  resemble the whole tree in distribution. More
precisely we have different types of distributions for the associated
tree, one type for each possible color of its root. The decomposition
of the associated tree into subtrees gives rise to a system of
distributional recurrences for the numbers of balls of each color. To
extract the asymptotic behavior from such systems we develop an
approach in the context of the contraction method.

The contraction method is well known in the probabilistic analysis of
algorithms. It was introduced by R\"osler \cite{Ro91} and first
developed systematically in Rachev and R\"uschendorf \cite{RaRu95}.
A rather general framework with numerous applications to the analysis
of recursive algorithms and random trees was given by Neininger and
R\"uschendorf \cite{NeRu04}. The contraction method has been used for
sequences of distributions of random variables (or random vectors or
stochastic processes) that satisfy an appropriate recurrence relation.
To the best of our knowledge it has not yet been used for systems of
such recurrence relations as they arise in the present paper, the only
exception being  Leckey et al. \cite{lenesz13} where tries are
analyzed under a Markov source model. A novel technical aspect of the present paper is that we extend the use of the contraction method to systems of recurrence relations systematically.

The aim of this paper is not to compete with other techniques with
respect to generality under which urn models can be analyzed. Instead
we discuss our approach at a couple of examples illustrating the
contraction framework in three frequently occurring asymptotic
regimes: normal limit laws, non-normal limit laws and regimes with
oscillating  distributional behavior.   We also discuss the case of
random entries in the replacement matrix. Our proofs are generic and
can easily be transferred to other urn models or be developed into
more general theorems when  asymptotic expansions of means
(respectively means and variances in the normal limit case) are
available, cf.~the types of expansions of the means in section
\ref{limiteq}.

A general assumption in the present paper is that the replacement
matrix is balanced, i.e., that we have $\sum_{j=1}^m a_{ij} =:K-1$ for
all $i=1,\ldots,m$, where $K\ge 2$ is a fixed integer. (The notation
$K$ is unfortunate since this integer is not random and mainly chosen
to have similarity in notation with earlier work on the contraction
method.) An  implication of the balance condition is that the  growths
of the subtrees of the associated tree processes can asymptotically
jointly be captured by Dirichlet distributions. This leads to
characterizations of  the limit distributions in all cases (normal,
non-normal and oscillatory behavior) by systems, cf.
(\ref{limitta})--(\ref{limittc}) below, of distributional fixed-point
equations where all coefficients are powers of components of a
Dirichlet distributed vector, see also the discussion in section
\ref{limiteq}.   The present
approach reveals that all three regimes are governed by  systems of distributional fixed-point equations of similar type.

The paper is organized as follows: in section \ref{appr}, we introduce
the associated trees into which the urn models are embedded and derive
the systems of distributional recurrences for the numbers of  balls of
a certain color from the  associated trees. In section \ref{limiteq},
we outline the types of systems of fixed-point equations that emerge
from the distributional recurrences after a proper normalization. To
make  these recurrences and fixed-point equations accessible to the
contraction method, in section \ref{secspaces}, we first introduce
spaces of probability distributions and appropriate cartesian product
spaces together with metrics on these product spaces. The metrics in
use are product versions of the minimal $L_p$ metrics and product
versions of the Zolotarev metrics. In section \ref{assofix}, we use
these spaces and metrics to show that our systems of distributional
fixed-point equations uniquely characterize vectors of probability
distributions  via a contraction property. These cover the types of
distributional fixed-point equations that appear in the  final section
\ref{conexp} where we discuss examples of limit laws for P\'olya urn
schemes within our approach. In  section \ref{conexp} also our
convergence proofs are worked out, again based on the product versions
of the  minimal $L_p$ and Zolotarev metrics. In section \ref{sec7} we compare our study of systems of recurrences with an alternative formulation based on multivariate recurrences and explain the advantages and necessity of our approach. 

For similar results see \cite{chmapo13} (announced after posting the present paper on {\tt arXiv.org}).\\

\noindent
{\bf Notation.} By $\stackrel{d}{\longrightarrow}$ convergence in
distribution is denoted. We denote the normal distribution on $\R$
with mean $\mu\in\R$ and variance $\sigma^2\ge 0$ by ${\cal
  N}(\mu,\sigma^2)$. In the case $\sigma^2=0$, this degenerates to the
Dirac measure in $\mu$.  Throughout the paper, the Bachmann-Landau
symbols are used in asymptotic statements. We denote by $\log(x)$ for $x>0$ the
natural logarithm of $x$ and the non-negative integers by $\N_{0}:=\{0,1,2,\dots\}$.\\

\noindent
{\bf Acknowledgements:} We thank two referees for their comments and careful reading. We also thank the e-Print archive {\tt arXiv.org} and
 Cornell University Library for making an electronic preprint of this work freely and publicly  available by January 16, 2013.

\section{A recursive description of P\'olya urns}\label{appr}
In this section, we explain our embedding of urn processes into
associated combinatorial  random tree structures growing in discrete
time. The distributional self-similarity within the subtrees of the
roots of these associated trees leads to systems of distributional
recurrences which constitute the core of our approach.  \\

\noindent
{\bf The P\'olya urn.}
To develop our approach, we first consider an urn model with two
colors, black and white, and a deterministic replacement matrix $R$. Below, an
extension of this approach to urns with more than two colors and
replacement matrices with random entries is discussed as well.
 To be definite, we use the replacement matrix
\begin{align}\label{rep2x2}
  R=\left[ \begin{array}{cc} a & b \\ c & d\end{array} \right] \quad \mbox{ with }
  a,d \in \N_0\cup\{-1\}\mbox{ and } b,c \in \N_0
\end{align}
with
\begin{align*}
  a+b=c+d=:K-1 \ge 1.
\end{align*}
The assumption that the sums of the entries in each row are
the same will become essential only from Lemma \ref{asysubsize} on. Now,  after drawing a black ball, this ball is placed back into the
urn together with $a$ new black balls and $b$ new white balls. If
drawing a white ball, it is placed back into the urn together with $c$
black balls and $d$ white balls. A diagonal entry $a=-1$ (or $d=-1$)
implies that a  drawn black (or white) ball is not placed back into
the urn while balls of the other color are still added to the urn. As
initial configuration, we consider both, one black ball or one white
ball. Other initial configurations can be dealt with as well, also
discussed below. We denote by $B_n^\mathrm{b}$ the number of black
balls after $n$ steps when initially starting with one black ball, by
$B_n^\mathrm{w}$ the number of black balls after $n$ steps when
initially starting with one white ball. Hence, we have
$B_0^\mathrm{b}=1$ and $B_0^\mathrm{w}=0$.\\

\noindent
{\bf The associated tree.}
We encode the urn process as follows by a discrete time evolution of a
random tree with nodes  colored black or white. This tree is called
{\em associated tree}. The initial  urn with one ball, say a black
one, is associated with a tree with one  root node of the same (black)
color. The ball in the urn is represented by this root node. Now
drawing the ball and placing it back into the urn together  with $a$
new black balls and $b$ new white balls is encoded in the associated
tree by adding $a+b+1=K$ children to the root node, $a+1$ of them
being black and $b$ being white.  The root node then  no longer
represents a ball  in the tree, whereas the $K$ new leaves of the tree
now represent the $K$ balls in the urn. Now, we iterate this
procedure: At any step, a ball is drawn from the urn. It is
represented by one of the leaves, say node $v$ in the tree. The urn
follows its dynamic. If the ball drawn is black, the (black) leaf  $v$
gets $K$ children, $a+1$ black ones and $b$ white ones. Similarly, if
the ball drawn is white, the (white) leaf $v$ gets $c$ black children
and $d+1$ white children. In both cases, $v$ no longer represents a
ball in the urn. The ball drawn and the new balls are represented by
the children of $v$. The correspondence between all other leaves of
the tree and the other balls in the urn remains unchanged. For an
example of an evolution of an urn and its associated tree see Figure~1.
\begin{figure}[htb] \label{fig1}
\noindent\rule[1ex]{\textwidth}{0.8pt}
\centering
\includegraphics{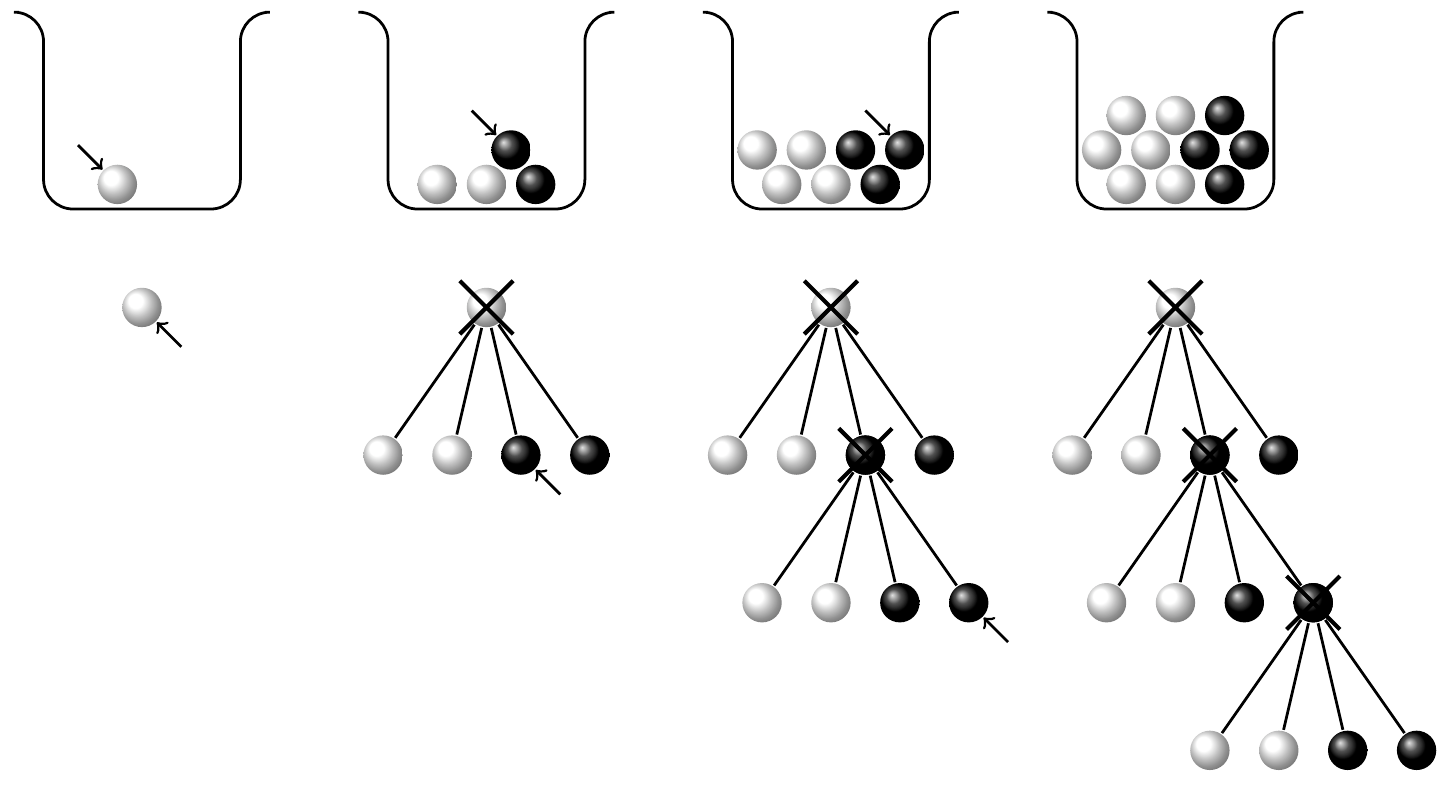}
\caption{A realization of the evolution of the P\'olya urn with
  replacement matrix~\usebox{\smlmat} and initially one white ball.
  The arrows indicate which ball is drawn (resp.~leaf is replaced) in
  each step. Below each urn its associated tree is shown. Leaf nodes
  correspond to the balls in the urn, non-leaf nodes (crossed out) do
  no longer correspond to balls in the urn. However, their color still
  matters for the recursive decomposition of the associated tree.}
\noindent\rule[1ex]{\textwidth}{0.8pt}
\end{figure}
Hence, at any time, the balls in the urn are represented by the leaves
of the associate tree, where the colors of balls and representing
leaves match. Each node of the tree is either a leaf or has $K$
children. We could as well simulate the urn process by only running
the evolution of the associated tree as follows: Start with one root
node of the color of the initial ball of the urn. At any step, choose
one of the leaves of the tree uniformly at random, inspect its color,
add $K$ children to the chosen leaf and color these children as
defined above. Then after $n$ steps, the tree has
$n\left(K-1\right)+1$ leaves. The number of black leaves is
distributed as $B_n^\mathrm{b}$ if the root node was black and
distributed as $B_n^\mathrm{w}$, if the root node was white.

Subsequently, it is important to note the following recursive structure
of the associated tree: For a fixed replacement matrix of the P\'olya
urn, we consider the two initial compositions of one black respectively
one white ball and their two associated trees. We call these the
$\mathrm{b}$-associated respectively $\mathrm{w}$-associated tree.
Consider one of these associated trees after $n\geq1$ steps.  It has
$n(K-1)+1$ leaves and each subtree rooted at a child of the associated
tree's root (we call them shortly only subtrees) has a random number
of leaves according to how often a leaf node has been chosen for
replacement in the subtree. We condition on the numbers of leaves of
the subtrees to be $i_r(K-1)+1$ with $i_r\in\N_0$ for $r=1,\ldots,K$.
Note that we have $\sum_{r=1}^K i_r=n-1$, the $-1$ resulting from the
fact that  in the first step of the evolution of the associated tree,
the subtrees are being generated, only afterwards they start growing.
From the evolution of the $\mathrm{b}$-associated tree, it is clear
that conditioned on the subtrees' numbers of leaves being
$i_r(K-1)+1$, the subtrees are stochastically independent and the
$r$-th subtree is distributed as an associated tree after $i_r$ steps.
Whether it has the distribution of the  $\mathrm{b}$- or the
$\mathrm{w}$-associated tree depends on the color of the subtree's
root node.

To summarize, we have that conditioned on their numbers of leaves, the
subtrees of associated trees are independent and distributed as
associated trees of corresponding size and type inherited from the
color of their root node. \\

\noindent {\bf System of recursive equations.} We set up recursive
equations for the distributions of the quantities $B_n^\mathrm{b}$ and
$B_n^\mathrm{w}$: For $B_n^\mathrm{b}$, we start the urn with one
black ball and get a $\mathrm{b}$-associated tree with a black root
node. Now, $B_n^\mathrm{b}$ is distributed as the number of black
leaves in the associated tree after $n$ steps which, for $n\geq1$, we
express as the sum of the numbers of black leaves of its subtrees. As
discussed above, conditionally on
$I^{(n)}=(I^{(n)}_1,\ldots,I^{(n)}_K)$, the vector of the numbers of
balls drawn in each subtree, these subtrees are independent and
distributed as $\mathrm{b}$-associated trees or
$\mathrm{w}$-associated trees of the corresponding size depending on
the color of their roots. In a $\mathrm{b}$-associated tree, the root
has $a+1$ black and $b=K-(a+1)$ white children. Hence, we obtain
\begin{align}\label{rde1}
  B_n^\mathrm{b} \stackrel{d}{=}
  \sum_{r=1}^{a+1} B^{\mathrm{b}, (r)}_{I^{(n)}_r}
  + \sum_{r=a+2}^{K} B^{\mathrm{w}, (r)}_{I^{(n)}_r}, \quad n\ge 1,
\end{align}
where $\stackrel{d}{=}$ denotes that left and right hand side have an
identical distribution, we have that $(B^{\mathrm{b}, (1)}_k)_{0\le k
  < n},\ldots, (B^{\mathrm{b}, (a+1)}_k)_{0\le k < n}$,
$(B^{\mathrm{w}, (a+2)}_k)_{0\le k < n},\ldots, (B^{\mathrm{w},
  (K)}_k)_{0\le k < n}$, $I^{(n)}$ are independent, the
$B^{\mathrm{b}, (r)}_k$ are distributed as $B^{\mathrm{b}}_k$, the
$B^{\mathrm{w}, (r)}_k$ are distributed as $B^{\mathrm{w}}_k$ for
$k=0,\ldots,n-1$ for the respective values of $r$.

Similarly, we obtain a recursive distributional equation for
$B_n^\mathrm{w}$. We have
\begin{align}\label{rde2}
B_n^\mathrm{w} \stackrel{d}{=}
\sum_{r=1}^{c} B^{\mathrm{b}, (r)}_{I^{(n)}_r}
+ \sum_{r=c+1}^{K} B^{\mathrm{w}, (r)}_{I^{(n)}_r}, \quad n\ge 1,
\end{align}
with conditions on independence and identical distributions as in
(\ref{rde1}). Note that with the initial value
$(B_0^\mathrm{b},B_0^\mathrm{w})=(1,0)$, the system of equations
(\ref{rde1})--(\ref{rde2}) defines the sequence of pairs of
distributions $({\cal L}(B_n^\mathrm{b}),{\cal
  L}(B_n^\mathrm{w}))_{n\ge 0}$.\\

\noindent
{\bf General number of colors.} The approach above for urns with two colors extends directly to
urns with an arbitrary number $m\ge 2$ of colors. We denote the replacement matrix by $R=(a_{ij})_{1\le i,j \le m}$ with
\begin{align*}
a_{ij} \in  \left\{ \begin{array}{cc}\N_0, &\mbox{for }i\neq j, \\ \N_0\cup \{-1\}, &\mbox{for }i=j,  \end{array} \right. \quad \mbox{ and } \quad \sum_{j=1}^m a_{ij} =:K-1\ge 1 \mbox{ for } i=1,\ldots,m.
\end{align*}
The colors (subsequently also called types) are now numbered $1,\ldots,m$ and we focus on the number of balls of type $1$  after $n$ steps. When starting with one ball of  type $j$ we denote by  $B_n^{[j]}$  the number of type $1$ balls after $n$ steps. To formulate a system of distributional recurrences generalizing (\ref{rde1}) and (\ref{rde2}) we further denote the intervals of integers
\begin{align}\label{intrep}
J_{ij}:= \left\{ \begin{array}{cl}
\left[1+ \sum_{k<i}a_{kj}, \sum_{k\le i} a_{kj}\right] \cap \N_0, & \mbox{for }i<j, \vspace{2mm} \\
\left[1+ \sum_{k<i}a_{kj}, 1+\sum_{k\le i} a_{kj}\right] \cap \N_0, & \mbox{for }i=j,\vspace{2mm}\\
\left[2+ \sum_{k<i}a_{kj}, 1+\sum_{k\le i} a_{kj}\right] \cap \N_0, & \mbox{for }i>j,
\end{array} \right.
\end{align}
with the convention $[x,y]=\emptyset$ if $x>y$. Then, we have
\begin{align}\label{rde3}
B_n^{[j]} \stackrel{d}{=} \sum_{i=1}^m \sum_{r\in J_{ij}} B^{[i], (r)}_{I^{(n)}_r}, \quad n\ge 1, \quad j\in \{1,\ldots,m\},
\end{align}
where, for each $j\in \{1,\ldots,m\}$ we have that  the
family
$$\left\{\left(B^{[i], (r)}_k\right)_{0\le k < n}\,\Big|\, r\in J_{ij}, i \in \{1,\ldots,m\}\right\} \cup \left\{I^{(n)}\right\}$$
is independent,  $B^{[i], (r)}_k$ is distributed as
$B^{[i]}_k$  for all $i \in \{1,\ldots,m\}$, $0\le k < n$
and $r\in J_{ij}$ and $I^{(n)}$ has the
distribution as above in Lemma  \ref{asysubsize}.\\

\noindent
{\bf Composition vectors.} For  urns with more than two colors one may study the numbers of balls of each color jointly. Even though the system (\ref{rde3}) gives only access to the marginals of this composition vector we could as well derive a system of recurrences for the composition vectors and develop our approach for the joint distribution of the composition vector. The work spaces $({\cal M}_s^{\R})^{\times d}$ and $({\cal M}_s^{\C})^{\times d}$ defined in section \ref{secspaces} below (there $d$ corresponds to the number of colors) then become $({\cal M}_s^{\R^{d-1}})^{\times d}$ and $({\cal M}_s^{\C^{d-1}})^{\times d}$. The Zolotarev metrics $\zeta_s$ and minimal $L_p$-metrics $\ell_p$ are defined on $\R^{d-1}$ and $\C^{d-1}$  as well and can be used to develop a similar limit theory for the composition vectors as presented here for their marginals. \\

\noindent
{\bf Random entries in the replacement matrix.} The case of a
replacement matrix with random entries such that each row almost
surely sums to a deterministic and fixed $K-1\ge 1$ can be covered by an extension of the system (\ref{rde3}). Instead of formulating such an extension explicitly, we discuss an example in section \ref{kersting}.\\

\noindent
{\bf Growth of subtrees.}
In our analysis, the asymptotic growth of the $K$  subtrees of the
associated tree is used.  We  denote by
$I^{(n)}=(I^{(n)}_1,\ldots,I^{(n)}_K)$ the vector of the numbers of
draws of leaves from each   subtree after $n\geq1$ draws in the full
associated tree. In other words, $I^{(n)}_r(K-1)+1$  is the number of
leaves of the $r$-th subtree after  $n\geq1$  steps. We have
$I^{(1)}=(0,\ldots,0)$,  and $I^{(2)}$ is a vector with all entries
being $0$ except for one coordinate which is $1$. To describe the
asymptotic growth of $I^{(n)}$, we need the Dirichlet distribution
$\mathrm{Dirichlet}((K-1)^{-1},\ldots,(K-1)^{-1})$: It is
the distribution of a random vector $(D_1,\ldots,D_K)$ with
$\sum_{r=1}^K D_r =1$ and such that $(D_1,\ldots,D_{K-1})$ has a
Lebesgue-density supported by the simplex ${\cal
  S}_K:=\bigl\{(x_1,\ldots,x_{K-1})\in [0,1]^{K-1}\,|\, \sum_{r=1}^{K-1}
x_r \le 1\bigr\}$ given for $x\in  {\cal S}_K$ by
\begin{align*}
  x=(x_1,\ldots,x_{K-1})\mapsto c_K
  \left(1-\sum_{r=1}^{K-1}x_{r}\right)^{\!\!\frac{2-K}{K-1}}
  \prod_{r=1}^{K-1}x_r^\frac{2-K}{K-1}
  ,\qquad c_K=\frac{\Gamma\bigl((K-1)^{-1}\bigr)^{1-K}}{K-1},
\end{align*}
where $\Gamma$ denotes Euler's Gamma function.
In particular, $D_1,\ldots,D_K$ are identically distributed with the
$\mathrm{beta}\bigl((K-1)^{-1},1\bigr)$ distribution, i.e., with
Lebesgue-density
\begin{align*}
x\mapsto\left(K-1\right)^{-1} x^\frac{2-K}{K-1}, \quad x \in [0,1].
\end{align*}
We have the following asymptotic behavior of $I^{(n)}$:
\begin{lem} \label{asysubsize}
Consider a P\'olya urn with constant row sum $K-1\ge 1$ and its
associated tree.  For the numbers of balls
$I^{(n)}=(I^{(n)}_1,\ldots,I^{(n)}_K)$ drawn in each subtree of the
associated tree when $n$ balls have been drawn in the whole associated
tree, we have, as $n\to\infty$,
\begin{align*}
  \left(\frac{I^{(n)}_1}{n},\ldots,\frac{I^{(n)}_K}{n}\right)
  \longrightarrow
  \left(D_1,\ldots,D_K\right)
\end{align*}
almost surely and in any $L_p$, where $(D_1,\ldots,D_K)$ has the
Dirichlet distribution
\begin{align*}
  {\cal L}(D_1,\ldots,D_K) =\mathrm{Dirichlet}\left(\frac{1}{K-1},\ldots,\frac{1}{K-1}\right).
\end{align*}
\end{lem}
\begin{proof}
The sequence $(I^{(n)}_1(K-1)+1,\ldots, I^{(n)}_K(K-1)+1)_{n\in\N_0}$
has an interpretation by another urn model, which we call the
subtree-induced urn: For this, we give additional labels to the
leaves of the associated tree. The set of possible labels is
$\{1,\ldots,K\}$ and we label a leaf $j$  if it belongs to the $j$-th
subtree of the root (any ordering of the subtrees of the root is
fine).  Hence, all  leaves of a subtree of the associated tree's root
get the same label, leaves of different subtrees get different labels.
Now, the subtree-induced urn has balls of colors $1,\ldots,K$. At any
time, the number of balls of each color is identical with the numbers
of leaves with the corresponding label. Hence, the dynamic of the
subtree-induced urn is that of a P\'olya urn with initially $K$
balls, one of each color.   Whenever a ball is drawn, it is placed
back into the urn together with $K-1$ balls of the same color. In
other words, the replacement matrix for the dynamic of the
subtree-induced urn is a $K\times K$ diagonal matrix with all diagonal
entries equal to $K-1$. After $n$ steps, we have  $I^{(n)}_r(K-1)+1$
balls of color $r$. The dynamic of the subtree-induced urn as a
$K$-color P\'olya-Eggenberger urn  is
well-known, cf.~Athreya \cite[Corollary~1]{Ath69}, we have for $n\to\infty$
\begin{align*}
  \left(\frac{I^{(n)}_1(K-1)+1}{n(K-1)+1},\ldots,\frac{I^{(n)}_K(K-1)+1}{n(K-1)+1}\right)
  \longrightarrow (D_1,\ldots,D_K)
\end{align*}
almost surely and in $L_{p}$ for any $p\geq1$, where $(D_1,\ldots,D_K)$ has a
$\mathrm{Dirichlet}((K-1)^{-1},\ldots,(K-1)^{-1})$ distribution. This
implies the assertion.
\end{proof}
Subsequently we only consider balanced urns such that we have the asymptotic behaviour of $I^{(n)}/n$ in Lemma \ref{asysubsize} available. The assumption of balance does only enter our subsequent analysis via Lemma \ref{asysubsize}. It seems feasible to apply our approach also to unbalanced urns that  have an associated tree
such that $I^{(n)}/n$ converges to a non-degenerate limit vector $V=(V_1,\ldots,V_K)$ of random probabilities, i.e.~of random $V_1,\ldots,V_K\ge 0$ such that $\sum_{r=1}^K V_r=1$ almost surely and $\Prob(\max_{1\le r\le K} V_r <1)>0$. It seems that the contraction argument may even allow that the distribution of $V$ depends on the color of the ball the urn is started with. We leave these issues for future research.

\section{Systems of limit equations} \label{limiteq}
In this section we outline how systems of the form (\ref{rde3}) are used subsequently. Based on the order of means and variances the $B^{[j]}_n$ are normalized and recurrences for the normalized random variables are considered. From this, with $n\to \infty$,  we derive systems of recursive distributional equations, see (\ref{limitta}), (\ref{fixtb}) and (\ref{limittc}). According to the general idea of the contraction method we then show first that these systems characterize distributions, see section \ref{assofix}, and second that the normalized random variables converge in distribution towards these distributions, see section \ref{conexp}. In the periodic case (c) we do not have convergence but the solution of the system (\ref{limittc}) allows to describe the asymptotic periodic behavior.

Crucial are the  expansions of the means
\begin{align*}
\mu^{[j]}_n:=\E\left[B_n^{[j]}\right], \quad j=1,\ldots,m,
\end{align*}
which are intimately related to the spectral decomposition of
the replacement matrix.
We only consider cases where these means grow linearly. Note however,
that even balanced urns can have quite different growth orders. An
example is the replacement matrix $\left[\begin{smallmatrix}4 & 0 \\ 3
    &1\end{smallmatrix} \right]$, see Kotz et al. \cite{komaro} for
this example or Janson \cite{Ja06T} for a comprehensive account of
urns with triangular replacement matrix.\\

\noindent
{\bf Type (a).} Assume  that we have expansions of the form, as $n\to\infty$,
\begin{align*}
\mu^{[j]}_n=c_\mu n + d_j n^\lambda + o(n^\lambda), \quad j=1,\ldots,m,
\end{align*}
with a constant $c_\mu>0$ independent of $j$, with constants
$d_j\in\R$ and an exponent $1/2<\lambda<1$. We call this scenario of
{\bf type (a)}. This suggests that the variances are of the order
$n^{2\lambda}$ and  a proper scaling is
\begin{align}\label{scala}
X^{[j]}_n:= \frac{B_n^{[j]} -\mu^{[j]}_n}{n^\lambda}, \quad n\ge 1,\quad j=1,\ldots,m.
\end{align}
Deriving from (\ref{rde3}) a system of recurrences for the $X^{[j]}_n$
and letting formally $n\to\infty$ (this is done explicitly in the
examples in section \ref{conexp}), we obtain the system of fixed-point
equations
\begin{align}\label{limitta}
X^{[j]}\stackrel{d}{=} \sum_{i=1}^m \sum_{r\in J_{ij}} D_r^\lambda X^{[i], (r)} + b^{[j]}, \quad j=1,\ldots,m,
\end{align}
where the $X^{[i], (r)}$ and $(D_1,\ldots,D_K)$ are independent, the
$X^{[i], (r)}$ are distributed as $X^{[i]}$, $(D_1,\ldots,D_K)$ is
distributed as in Lemma \ref{asysubsize} and the $b^{[j]}$ are
functions of $(D_1,\ldots,D_K)$. It turns out that such a system
subject to centered $X^{[j]}$ with finite second moments has a unique
solution on the level of distributions (Theorem \ref{fp1}). This
identifies the weak
limits of the $X^{[j]}_n$. Examples are in sections \ref{ex1sec} and \ref{kersting}. One can as well obtain the same system (\ref{limitta}) with $b^{[j]}=0$ for all $j$ by only centering the $B_n^{[j]}$ by $c_\mu n$  instead of the exact mean. Then, system  (\ref{limitta}) has to be solved subject to finite second moments and appropriate means. Moreover, the system allows to calculate higher order moments of the solution. From the second and third moments one can typically see that the solution is not a vector of normal distributions.

Expansions of the form
\begin{align*}
\mu^{[j]}_n=c_\mu n + d_j n^\lambda \log^\nu(n) + o(n^\lambda\log^\nu(n)), \quad j=1,\ldots,m,
\end{align*}
with $\nu\ge 1$ also appear, see Janson \cite{Ja04} or the table on page 279 of Pouyanne \cite{Pou05} for a classification. Such additional  factors $\log^\nu(n)$, slowly varying at infinity, give rise to the same limit system (\ref{limitta}) and hence do not affect the limit distributions. These cases can be covered similarly to the examples in section \ref{conexp}. We omit the details; see however Hwang and Neininger \cite{hwne02} for  the occurrence and analysis of
similar slowly varying factors. \\

\noindent
{\bf Type (b).} Assume  that we have expansions of the form, as $n\to\infty$,
\begin{align*}
\mu^{[j]}_n=c_\mu n +  o(\sqrt{n}), \quad j=1,\ldots,m,
\end{align*}
with a constant $c_\mu>0$ independent of $j$. We call this scenario of {\bf type (b)}. This suggests that the variances are of  linear order  and  a proper scaling is
\begin{align}\label{scalb}
X^{[j]}_n:= \frac{B_n^{[j]} -\mu^{[j]}_n}{\sqrt{\Var(B_n^{[j]})}}, \quad n\ge 1,\quad j=1,\ldots,m
\end{align}
(or $\sqrt{\Var(B_n^{[j]})}$ replaced by $\sqrt{n}$).
The corresponding system of fixed-point equations in the limit is
\begin{align}\label{fixtb}
X^{[j]}\stackrel{d}{=} \sum_{i=1}^m \sum_{r\in J_{ij}} \sqrt{D_r} X^{[i], (r)}, \quad j=1,\ldots,m,
\end{align}
with conditions as in (\ref{limitta}). Under appropriate assumptions on moments we find that the only solution is all $X^{[j]}$ being standard normally distributed (Theorem \ref{fp2}). This leads to asymptotic normality of the $X^{[j]}_n$. Examples are given in sections \ref{ex1sec} and \ref{kersting}. The  case
\begin{align*}
\mu^{[j]}_n=c_\mu n +  \Theta(\sqrt{n}), \quad j=1,\ldots,m,
\end{align*}
leads to the same system of fixed-point equations (\ref{fixtb}). However, here the variances typically are of order $n \log^\delta(n)$ with a positive $\delta$.\\

\noindent
{\bf Type (c).} Assume  that we have expansions of the form, as $n\to\infty$,
\begin{align*}
\mu^{[j]}_n=c_\mu n +  \Re\left(\kappa_j n^{\mathrm{i} \mu}\right) n^\lambda + o(n^\lambda), \quad j=1,\ldots,m,
\end{align*}
with a constant $c_\mu>0$ independent of $j$, $1/2<\lambda<1$,
constants $\kappa_j\in \C$ and $\mu \in \R\setminus\{0\}$. (By
$\mathrm{i}$ the imaginary unit is denoted.) We call this scenario of
{\bf type (c)}. This suggests oscillating variances of the order
$n^{2\lambda}$. The oscillatory behavior of mean and variance can
typically not be removed by proper scaling to obtain convergence
towards a limit distribution. Using the scaling
\begin{align} \label{scalc}
X^{[j]}_n:= \frac{B_n^{[j]} -c_\mu n}{n^\lambda}, \quad n\ge 1,\quad j=1,\ldots,m.
\end{align}
it turns out that the oscillating behavior of the  $X^{[j]}_n$ can be captured by the
system of fixed-point equations
\begin{align}\label{limittc}
X^{[j]}\stackrel{d}{=} \sum_{i=1}^m \sum_{r\in J_{ij}} D_r^\omega X^{[i], (r)}, \quad j=1,\ldots,m,
\end{align}
with conditions as in (\ref{limitta}) and $\omega:=\lambda + \mathrm{i}\mu$.
Under appropriate moment assumptions this has a unique solution within distributions
on $\C$ (Theorem \ref{fp3}). An example of a corresponding distributional
approximation is given in section  \ref{excyc}.

As in type (a) we may have additional factors $\log^\nu(n)$, i.e.
\begin{align*}
\mu^{[j]}_n=c_\mu n +  \Re\left(\kappa_j n^{\mathrm{i} \mu}\right) n^\lambda \log^\nu(n)+ o(n^\lambda\log^\nu(n)), \quad j=1,\ldots,m.
\end{align*}
The  comments as for type (a) cases above apply here as well.\\

Note that the approach of  embedding urn models into continuous time multitype branching processes, see  \cite{atka68, Ja04}, also leads to characterizations of the limit distributions as in
(\ref{limitta}) and (\ref{limittc}). However, the  form of the fixed-point equations is different, see the system in equation  (3.5) in Janson \cite{Ja04}. Properties of such fixed-points have been studied in Chauvin et al. \cite{chposa11,chlipo12,{chlipo12b}}.

\section{Spaces of distributions and metrics} \label{secspaces}
In this section we define cartesian products of spaces of probability distributions and metrics on these products. These metric spaces will be used below to first characterize limit distributions of urn models (section \ref{assofix}) and then  prove convergence in distribution of
  the scaled numbers of balls of a color (section \ref{conexp}). \\

\noindent
{\bf Spaces.}
We denote by ${\cal M}^\R$ the space of all probability distributions on $\R$ with the Borel $\sigma$-field. Moreover, we consider the subspaces
\begin{align*}
{\cal M}^\R_s&:=\left\{{\cal L}(X)\in {\cal M}^\R  \,\Big|\, \E[|X|^s]<\infty\right\}, \quad s>0,\\
{\cal M}^\R_s(\mu)&:=\left\{{\cal L}(X)\in {\cal M}^\R_s \,\Big|\, \E[X]=\mu \right\}, \quad s\ge 1,\mu\in\R \\
{\cal M}^\R_s(\mu,\sigma^2)&:=\left\{{\cal L}(X)\in {\cal M}^\R_s(\mu) \,\Big|\, \Var(X)=\sigma^2 \right\}, \quad s\ge 2,\mu\in\R, \sigma\ge 0.
\end{align*}
We need the $d$-fold cartesian products, $d\in\N$, of these spaces denoted by
\begin{align}\label{defspace}
\left({\cal M}^\R_s\right)^{\times d}:= {\cal M}^\R_s \times \cdots \times {\cal M}^\R_s,
\end{align}
and analogously $({\cal M}^\R_s(\mu))^{\times d}$ and $({\cal M}^\R_s(\mu,\sigma^2))^{\times d}$.

We also need probability distributions on the complex plane $\C$. By ${\cal M}^\C$ the space of all probability distributions on $\C$ with the Borel $\sigma$-field is denoted. Moreover, for $\gamma\in \C$ we use the subspaces and  product space
\begin{align*}
{\cal M}^\C_s&:=\left\{{\cal L}(X)\in {\cal M}^\C  \,\Big|\, \E[|X|^s]<\infty\right\}, \quad s>0,\\
{\cal M}^\C_2(\gamma)&:=\left\{{\cal L}(X)\in {\cal M}^\C_2 \,\Big|\,  \E[X]=\gamma\right\},\\
\left({\cal M}^\C_2(\gamma)\right)^{\times d}&:= {\cal M}^\C_2(\gamma) \times \cdots \times {\cal M}^\C_2(\gamma).
\end{align*}

To cover the different behavior of the urns two types of metrics  are constructed, extensions of the Zolotarev metrics $\zeta_s$ and the minimal $L_p$-metric $\ell_p$ to the product spaces defined above.  \\

\noindent
{\bf Zolotarev metric.}
 The Zolotarev metric  has been introduced and studied in  Zolotarev \cite{zo76,zo77}.
The contraction method based on the Zolotarev metric was systematically developed  in \cite{NeRu04} and, for issues that go beyond what is needed in this paper, in \cite{KaJa12} and \cite{NeSu12}. We only need
the following properties:
For distributions
${\cal L}(X)$, ${\cal L}(Y)\in {\cal M}^\R$  the Zolotarev distance $\zeta_s$, $s>0$, is defined by
\begin{equation}
\label{eq:3.6} \zeta_s(X,Y) := \zeta_s({\cal L}(X),{\cal L}(Y)):=\sup_{f\in {\cal F}_s}|\E[f(X) -
f(Y)]|
\end{equation}
where $s=m+\alpha$ with $0<\alpha\le 1$,
$m\in\N_0$, and
\begin{equation}
{\cal F}_s:=\{f\in
C^m(\R,\R):|f^{(m)}(x)-f^{(m)}(y)|\le
|x-y|^\alpha\},
\end{equation}
 the space of $m$ times
continuously differentiable functions from
$\R$ to $\R$ such that the $m$-th
derivative is H\"older continuous of order
$\alpha$ with H\"older-constant $1$.

We have that $\zeta_s(X,Y)<\infty$, if all
moments of orders $1,\ldots,m$ of $X$ and $Y$ are
equal and if the $s$-th absolute moments of $X$ and
$Y$ are finite.  Since later on  the cases $1<s\le 3$ are
used, we have two basic cases: First, for $1<s\le 2$ we have $\zeta_s(X,Y)<\infty$  for          ${\cal L}(X)$, ${\cal L}(Y)\in {\cal M}^\R_s(\mu)$ for any $\mu\in\R$. Second, for $2<s\le 3$ we have $\zeta_s(X,Y)<\infty$  for  ${\cal L}(X)$, ${\cal L}(Y)\in {\cal M}^\R_s(\mu,\sigma^2)$ for any $\mu\in\R$ and $\sigma\ge 0$. Moreover, the pairs $({\cal M}^\R_s(\mu), \zeta_s)$ for $1<s\le 2$ and $({\cal M}^\R_s(\mu,\sigma^2), \zeta_s)$ for $2<s\le 3$ are complete metric spaces; for the completeness see \cite[Theorem 5.1]{DrJaNe08}.

 Convergence
in $\zeta_s$ implies weak convergence on $\R$. Furthermore,  $\zeta_s$ is $(s,+)$ ideal, i.e., we have
\begin{align}\label{ideal}
\zeta_s(X+Z,Y+Z)\le\zeta_s(X,Y), \quad  \zeta_s(cX,cY) = c^s \zeta_s(X,Y)
\end{align}
for all  $Z$ being independent of $(X,Y)$ and all $c>0$. Note, that this implies  for $X_1,\ldots,X_n$ independent and  $Y_1,\ldots,Y_n$ independent such that the respective $\zeta_s$ distances are finite that
\begin{align}\label{sumabsch}
\zeta_s\left(\sum_{i=1}^n X_i, \sum_{i=1}^n Y_i\right)\le \sum_{i=1}^n \zeta_s(X_i,Y_i).
\end{align}
On the product spaces $({\cal M}^\R_s(\mu))^{\times d}$ for $1<s\le 2$ and $({\cal M}^\R_s(\mu,\sigma^2))^{\times d}$  for $2<s\le 3$ our first main tool is
\begin{align*}
\zeta_s^\vee((\nu_1,\ldots,\nu_d), (\mu_1,\ldots,\mu_d)):= \max_{1\le j\le d} \zeta_s(\nu_j,\mu_j),
\end{align*}
where $(\nu_1,\ldots,\nu_d), (\mu_1,\ldots,\mu_d) \in {\cal M}^\R_s(\mu))^{\times d}$ and  $\in ({\cal M}^\R_s(\mu,\sigma^2))^{\times d}$ respectively.
Note that $\zeta_s^\vee$ is a complete metric on the respective product spaces and induces the product topology. \\

\noindent
{\bf Minimal $\mathbf{L_p}$-metric $\mathbf{\ell_p}$.} First for probability metrics on the real line the  minimal $L_p$-metric $\ell_p$, $1\le p<\infty$ is defined by
\begin{align*}
\ell_p(\nu,\varrho):=\inf\{\|V-W\|_p\,|\, {\cal L}(V)=\nu, {\cal L}(W)=\varrho\},\quad \nu,\varrho \in {\cal M}^\R_p,
\end{align*}
where $\|V-W\|_p:=(\E[|V-W|^p])^{1/p}$ is the usual $L_p$-norm. The spaces $({\cal M}_p, \ell_p)$  and $({\cal M}_p(\mu),\ell_p)$ for $1\le p<\infty$ are complete metric spaces, see \cite{bifr}. The infimum in the definition of $\ell_p$ is a minimum. Random variables $V'$, $W'$ with distributions $\nu$ and $\varrho$ respectively such that $\ell_p(\nu,\varrho)=\|V'-W'\|_p$ are called {\em optimal couplings}. They do exist for all $\nu,\varrho \in {\cal M}^\R_1$. We use the notation $\ell_p(X,Y):=\ell_p({\cal L}(X),{\cal L}(Y))$ for random variables $X$ and $Y$.
Subsequently also the following inequality between the $\ell_p$ and $\zeta_s$ metrics is used:
\begin{align}\label{zetaellest}
\zeta_s(X,Y)\le \left( \left(\E\left[|X|^s\right]\right)^{1-1/s}+\left(\E\left[|Y|^s\right]\right)^{1-1/s}\right) \ell_s(X,Y), \quad 1< s\le 3,
\end{align}
where, for $1<s\le 2$, we need ${\cal L}(X),  {\cal L}(Y) \in {\cal M}^\R_s(\mu)$ for some $\mu\in\R$ and, for  $2<s\le 3$, we need ${\cal L}(X),  {\cal L}(Y) \in {\cal M}^\R_s(\mu,\sigma^2)$ for some $\mu\in\R$ and $\sigma\ge 0$, see \cite[Lemma 5.7]{DrJaNe08}.

On the product space $({\cal M}_2^\R(0))^{\times d}$ we define
\begin{align*}
\ell_2^\vee((\nu_1,\ldots,\nu_{d}),(\varrho_1,\ldots,\varrho_{d})):=
\max_{1\le j\le d} \ell_2(\nu_j, \varrho_j),
\end{align*}
where $(\nu_1,\ldots,\nu_d), (\mu_1,\ldots,\mu_d) \in ({\cal M}^\R_2(0))^{\times d}$. Note that
$({\cal M}^\R_2(0))^{\times d}, \ell_2^\vee)$ is a complete metric space as well.

Second,  on the complex plane the  minimal $L_p$-metric $\ell_p$ is defined  similarly by
\begin{align*}
\ell_p(\nu,\varrho):=\inf\{\|V-W\|_p\,|\, {\cal L}(V)=\nu, {\cal L}(W)=\varrho\},\quad \nu,\varrho \in {\cal M}^\C_p,
\end{align*}
with the analogous definition of the $L_p$-norm. The respective metric spaces are complete as in the real case and optimal couplings exist as well. On the product space $({\cal M}_2^\C(0))^{\times d}$ we use
\begin{align*}
\ell_2^\vee((\nu_1,\ldots,\nu_{d}),(\varrho_1,\ldots,\varrho_{d})):=
\max_{1\le j\le d} \ell_2(\nu_j, \varrho_j),
\end{align*}
where $(\nu_1,\ldots,\nu_d), (\mu_1,\ldots,\mu_d) \in ({\cal M}^\C_2(0))^{\times d}$. Note that
$({\cal M}^\C_2(0))^{\times d}, \ell_2^\vee)$ is a complete metric space as well. \\

\noindent
{\bf Preview on the use of spaces and metrics.}
The guidance on which space and metric to use in which asymptotic regime of  P\'olya urns is as follows. We come back to the  three types (a)--(c) of urns from the previous section:
\begin{itemize}
\item[(a)] Urns that after scaling lead to convergence  to a non-normal limit distribution. Typically
 such a  convergence holds almost surely, however we only discuss convergence in distribution.
\item[(b)] Urns that after scaling lead to convergence  to a normal limit. Such a convergence typically does not hold almost surely, but at least in distribution.
\item[(c)] Urns that even after a proper scaling do not lead to convergence. Instead there is an asymptotic oscillatory behavior of the distributions. Such  oscillatory behavior can even be captured almost surely,  we discuss a (weak) description for distributions.
\end{itemize}
The cases of {\bf type (a)} can be dealt with on the space $({\cal M}^\R_2(\mu))^{\times d}$ with appropriate $\mu\in\R$ and $d\in\N$, where, by centering, one can always achieve the choice $\mu=0$. One can either use the metric $\zeta_2^\vee$ or $\ell_2^\vee$ which lead to similar results, although based on different details in the  proofs. We will only present the use of $\zeta_2^\vee$, since we then can easily extend the argument also to the type (b) cases by switching from $\zeta_2^\vee$ to $\zeta_3^\vee$. This leads to a  more concise presentation. However, the $\ell_2^\vee$ metric appears to be equally convenient to apply in type (a) cases to us.

The cases of {\bf type (b)} can  be dealt with on the space $({\cal M}^\R_s(\mu,\sigma^2))^{\times d}$ with $2<s\le 3$ and appropriate $\mu\in\R$, $\sigma>0$ and $d\in\N$. By normalization, one can always achieve the choices $\mu=0$ and $\sigma=1$. Since in the context of urns third absolute moments in type (b) cases typically do exist, one can use $s=3$ and the metric $\zeta_3^\vee$. We do not know how to use the $\ell_p^\vee$ metrics in type (b) cases.

The cases of {\bf type (c)}  can  be dealt with on the space $({\cal M}^\C_2(\gamma))^{\times d}$  with appropriate $\gamma\in\R$ and $d\in\N$. The metric used subsequently in  type (c) cases is the complex version of $\ell_2^\vee$. In our example below we will however use ${\cal M}^\C_2(\gamma_1) \times \cdots \times {\cal M}^\C_2(\gamma_d)$ with $\gamma_1,\ldots,\gamma_d\in \C$ to be able to work with a more natural scaling of the random variables, the metric still being $\ell_2^\vee$. We think that also $\zeta_2^\vee$ can be used in type (c) cases but did not check the details since the application of $\ell_2^\vee$ is straightforward.

\section{Associated fixed point equations}\label{assofix}
We fix $d,d'\in \N$, a $d\times d'$ matrix $(A_{ir})$ of random
variables and a vector $(b_1,\ldots,b_d)$ of random variables. Either
all of these random variables are real or all of them are complex. Furthermore, we are given a $d\times d'$ matrix $(\pi(i,r))$ with all entries $\pi(i,r)\in\{1,\ldots,d\}$. First, we consider the case, where all $A_{ir}$ and all $b_i$ are real. We associate a map
\begin{align}
T: \left({\cal M}^{\R}\right)^{\times d}
&\to \left({\cal M}^{\R}\right)^{\times d} \nonumber\\
(\mu_1,\ldots,\mu_{d}) &\mapsto (T_1(\mu_1,\ldots,\mu_{d}),\ldots, T_{d}(\mu_1,\ldots,\mu_{d})) \label{defT1} \\
T_i(\mu_1,\ldots,\mu_{d}) &:= {\cal L}\left(\sum_{r=1}^{d'} A_{ir} Z_{ir} + b_i\right)
\label{defT2}
\end{align}
with $(A_{i1},\ldots,A_{id'},b_i)$, $Z_{i1},\ldots,Z_{id'}$ independent, and $Z_{ir}$ distributed as $\mu_{\pi(i,r)}$ for $r=1,\ldots,d'$ and all components $i=1,\ldots,d$.

In the case, where the $A_{ir}$  and $b_i$ are complex random variables, we define
a map $T'$ similar to $T$:
\begin{align}
T': \left({\cal M}^{\C}\right)^{\times d}
&\to \left({\cal M}^{\C}\right)^{\times d} \label{defT'}\\
(\mu_1,\ldots,\mu_{d}) &\mapsto (T'_1(\mu_1,\ldots,\mu_{d}),\ldots, T'_{d}(\mu_1,\ldots,\mu_{d}))  \nonumber
\end{align}
with $T'_i(\mu_1,\ldots,\mu_{d})$ defined as for $T_i$ in (\ref{defT2}).

For the three regimes discussed in the preview within section \ref{secspaces} we use the following three theorems (Theorem \ref{fp1} for type (a), Theorem \ref{fp2} for type (b), and Theorem \ref{fp3} for type (c)) on existence of fixed-points of $T$ and $T'$.

\begin{thm} \label{fp1} Assume that in the definition of $T$ in (\ref{defT1}) and (\ref{defT2}) the $A_{ir}$ and $b_i$ are square integrable real random variables
with  $\E[b_i]=0$ for all $1\le i\le d$ and $1\le r\le d'$ and
\begin{align} \label{contcond1}
\max_{1\le i\le d} \sum_{r=1}^{d'} \E\left[A_{ir}^2\right]&<1.
\end{align}
Then the restriction of $T$ to $({\cal M}^{\R}_2(0))^{\times d}$ has a unique fixed-point.
\end{thm}

\begin{thm} \label{fp2}Assume that in the definition of $T$ in (\ref{defT1}) and (\ref{defT2}) for some $\varepsilon>0$ the $A_{ir}$ are $L_{2+\varepsilon}$-integrable real random variables
and $b_i=0$ for all $1\le i\le d$ and $1\le r\le d'$, that almost surely
\begin{align}\label{normcoe}
 \sum_{r=1}^{d'} A_{ir}^2=1 \quad \mbox{ for all } i=1,\ldots,d
\end{align}
and
\begin{align}\label{techcon}
 \min_{1\le i\le d} \Prob\left(\max_{1\le r\le d'} |A_{ir}| <1\right) >0.
\end{align}
Then, for all $\sigma^2\ge 0$, the restriction of $T$ to $({\cal M}^{\R}_{2+\varepsilon}(0,\sigma^2))^{\times d}$
has the unique fixed-point $({\cal N}(0,\sigma^2),\ldots,{\cal N}(0,\sigma^2))$.
\end{thm}

\begin{thm} \label{fp3}Assume that in the definition of $T'$ in
  (\ref{defT'}) the $A_{ir}$ and $b_i$ are square integrable complex
  random variables for all $1\le i\le d$ and $1\le r\le d'$ and that
  for $\gamma_1,\ldots,\gamma_d\in\C$ we have
  \begin{align}\label{cccent}
 \E[b_i] + \sum_{r=1}^{d'} \gamma_{\pi(i,r)}\E[A_{ir}]&=\gamma_i, \quad i=1,\ldots,d.
\end{align}
If moreover
\begin{align}\label{cccont}
\max_{1\le i\le d} \sum_{r=1}^{d'} \E\left[|A_{ir}|^2\right]<1
\end{align}
then the restriction of $T'$ to ${\cal M}^{\C}_2(\gamma_1)\times \cdots \times {\cal M}^{\C}_2(\gamma_d) $ has a unique fixed-point.
\end{thm}
Note that a special case of Theorem \ref{fp1} was used in the proof of  \cite[Theorem 3.9 (iii)]{Ja04} with a similar proof technique as in our proof of Theorem \ref{fp3}.

The rest of this section contains the proofs of Theorems \ref{fp1} -- \ref{fp3}.

\noindent
\begin{proof} {\em (Theorem \ref{fp1}).}
First note that for $(\mu_1,\ldots,\mu_d)\in ({\cal M}^{\R}_2(0))^{\times d}$,
by independence in definition (\ref{defT2}) and $\E[b_i]=0$ we have
$T_i(\mu_1,\ldots,\mu_d) \in {\cal M}^{\R}_2(0)$ for $i=1,\ldots,d$.
Hence, the restriction of $T$ to $({\cal M}^{\R}_2(0))^{\times d}$ maps
into $({\cal M}^{\R}_2(0))^{\times d}$.

Next, we show that the restriction of  $T$ to $({\cal M}^{\R}_2(0))^{\times d}$
is a (strict) contraction with respect to the metric $\zeta_2^\vee$: For
$(\mu_1,\ldots,\mu_d), (\nu_1,\ldots,\nu_d)\in ({\cal M}^{\R}_2(0))^{\times d}$
we first fix $i\in\{1,\ldots,d\}$. Let  $Z_{i1},\ldots,Z_{id'}$ and $Z'_{i1},\ldots,Z'_{id'}$ be real random variables
such that   $Z_{ir}$ is distributed as $\mu_{\pi(i,r)}$ and  $Z'_{ir}$ is distributed
as $\nu_{\pi(i,r)}$. Moreover, assume that
both families $\{(A_{i1},\ldots,A_{id'},b_i)$, $Z_{i1},\ldots,Z_{id'}\}$
and $\{(A_{i1},\ldots,A_{id'},b_i)$, $Z'_{i1},\ldots,Z'_{id'}\}$ are independent.
Then we have
\begin{align}\label{darim}
T_i(\mu_1,\ldots,\mu_d) = {\cal L}\left(\sum_{r=1}^{d'} A_{ir}Z_{ir} +b_i\right), \quad
T_i(\nu_1,\ldots,\nu_d) = {\cal L}\left(\sum_{r=1}^{d'} A_{ir}Z'_{ir} +b_i\right).
\end{align}
Conditioning on $(A_{i1},\ldots,A_{id'},b_i)$ and denoting this vector's distribution by $\Upsilon$ we obtain
\begin{align}\lefteqn{
    \zeta_2(T_i(\mu_1,\ldots,\mu_d),T_i(\nu_1,\ldots,\nu_d))} \nonumber\\
  &= \sup_{f\in{\cal F}_2}\biggl|\int{
    \Erw{
      f\Bigl(\sum_{r=1}^{d'}\alpha_r Z_{ir}+\beta\Bigr)
      -f\Bigl(\sum_{r=1}^{d'}\alpha_r Z'_{ir}+\beta\Bigr)}
  }\,d\Upsilon(\alpha_1,\ldots,\alpha_{d'},\beta)\biggr|\nonumber\\
  &\leq \int{ \sup_{f\in{\cal F}_2}
    \Erw{\biggl|
      f\biggl(\sum_{r=1}^{d'}\alpha_r Z_{ir}+\beta\biggr)
      -f\biggl(\sum_{r=1}^{d'}\alpha_r Z'_{ir}+\beta\biggr)
      \biggr|}
  }\,d\Upsilon(\alpha_1,\ldots,\alpha_{d'},\beta)\nonumber\\
   &=\int \zeta_2\left(\sum_{r=1}^{d'}\alpha_r Z_{ir} +\beta,\sum_{r=1}^{d'}\alpha_r Z'_{ir} +\beta\right) \,d\Upsilon(\alpha_1,\ldots,\alpha_{d'},\beta) \label{fp1e1}
\end{align}
Since $\zeta_2$ is $(2,+)$-ideal, we obtain from (\ref{ideal}) that
$\zeta_2(\sum \alpha_r Z_{ir} +\beta,\sum \alpha_r Z'_{ir} +\beta)\le
\sum \alpha_r^2 \zeta_2(Z_{ir},Z'_{ir})$. Hence, we can further
estimate
\begin{align}
\lefteqn{
\zeta_2(T_i(\mu_1,\ldots,\mu_d),T_i(\nu_1,\ldots,\nu_d))}\nonumber\\
&\le \int \sum_{r=1}^{d'} \alpha_r^2 \zeta_2(Z_{ir},Z'_{ir})\,d\Upsilon(\alpha_1,\ldots,\alpha_{d'},\beta)\nonumber\\
&= \int \sum_{r=1}^{d'} \alpha_r^2 \zeta_2(\mu_{\pi(i,r)},\nu_{\pi(i,r)})
\,d\Upsilon(\alpha_1,\ldots,\alpha_{d'},\beta)\nonumber\\
&\le \left(\sum_{r=1}^{d'} \E\left[A_{ir}^2\right]\right) \zeta_2^\vee((\mu_1,\ldots,\mu_d),(\nu_1,\ldots,\nu_d)). \label{fp1e2}
\end{align}
Now, taking the maximum over $i$ yields
\begin{align}
  \zeta_2^\vee(T(\mu_1,\ldots,\mu_d),T(\nu_1,\ldots,\nu_d)) \le
  \left(\max_{1\le i\le d}\sum_{r=1}^{d'}
    \E\left[A_{ir}^2\right]\right)
  \zeta_2^\vee((\mu_1,\ldots,\mu_d),(\nu_1,\ldots,\nu_d)).\label{fp1e3}
\end{align}
Hence, condition (\ref{contcond1}) implies that the restriction of $T$
to $({\cal M}^{\R}_2(0))^{\times d}$ is a contraction. Since the
metric $\zeta_2^\vee$ is complete, Banach's fixed-point theorem
implies the assertion.
\end{proof}

\noindent
\begin{proof} {\em (Theorem \ref{fp2}).}
This proof is similar to the previous proof of Theorem \ref{fp1}. Let $\varepsilon>0$ be as in Theorem \ref{fp2} and $\sigma>0$ be arbitrary.  First note that for $(\mu_1,\ldots,\mu_d)\in ({\cal M}^{\R}_{2+\varepsilon}(0,\sigma^2))^{\times d}$,
by independence in definition (\ref{defT2}), condition (\ref{normcoe}), and $b_i=0$ we have
$T_i(\mu_1,\ldots,\mu_d) \in {\cal M}^{\R}_{2+\varepsilon}(0,\sigma^2)$ for $i=1,\ldots,d$.
Hence, the restriction of $T$ to $({\cal M}^{\R}_{2+\varepsilon}(0,\sigma^2))^{\times d}$ maps
into $({\cal M}^{\R}_{2+\varepsilon}(0,\sigma^2))^{\times d}$.

We set $s:=(2+\varepsilon) \wedge 3$.
For
$(\mu_1,\ldots,\mu_d), (\nu_1,\ldots,\nu_d)\in ({\cal M}^{\R}_{2+\varepsilon}(0,\sigma^2))^{\times d}$  we choose $Z_{i1},\ldots,Z_{id'}$ and $Z'_{i1},\ldots,Z'_{id'}$ as in the proof of Theorem \ref{fp1} such that we have (\ref{darim}). Note that with our choice of $s$ we have $\zeta_s(T_i(\mu_1,\ldots,\mu_d),T_i(\nu_1,\ldots,\nu_d))<\infty$. With an estimate analogous to (\ref{fp1e1}) -- (\ref{fp1e3}), using now that $\zeta_s$ is $(s,+)$-ideal, we obtain
\begin{align*}
\zeta_s^\vee(T(\mu_1,\ldots,\mu_d),T(\nu_1,\ldots,\nu_d))
\le \left(\max_{1\le i\le d}\sum_{r=1}^{d'} \E\left[|A_{ir}|^s\right]\right) \zeta_s^\vee((\mu_1,\ldots,\mu_d),(\nu_1,\ldots,\nu_d)).
\end{align*}
Note that $s>2$ and the conditions (\ref{normcoe}) and (\ref{techcon}) imply that $\sum_{r=1}^{d'} \E[|A_{ir}|^s]<1$ for all $i=1,\ldots,d$.
Hence,  the restriction of  $T$ to $({\cal M}^{\R}_{2+\varepsilon}(0,\sigma^2))^{\times d}$
is a contraction and the completeness of $\zeta_s^\vee$ implies the existence of a unique fixed-point. With the convolution property ${\cal N}(0,\sigma_1^2)\ast {\cal N}(0,\sigma_2^2)={\cal N}(0,\sigma_1^2+\sigma_2^2)$ for $\sigma_1,\sigma_2\ge 0$ one can directly check that
$({\cal N}(0,\sigma^2),\ldots,{\cal N}(0,\sigma^2))$ is a fixed-point of $T$ in $({\cal M}^{\R}_{2+\varepsilon}(0,\sigma^2))^{\times d}$.
\end{proof}

\begin{proof}{\em (Theorem \ref{fp3}).}
Let $\gamma_1,\ldots,\gamma_d$ be as in Theorem \ref{fp3} and abbreviate ${\cal P}:={\cal M}^{\C}_2(\gamma_1)\times \cdots \times {\cal M}^{\C}_2(\gamma_d)$.
First note that for $(\mu_1,\ldots,\mu_d)\in {\cal P}$ from independence in  the definition of $T'_i(\mu_1,\ldots,\mu_{d})$ and the finite second moments of the $A_{ir}$ and $b_i$ we obtain $T'_i(\mu_1,\ldots,\mu_{d})\in {\cal M}^\C_2$ for all $i=1,\ldots,d$. For a random variable $W$ with distribution $T'_i(\mu_1,\ldots,\mu_{d})$ we have
\begin{align*}
\E[W]=\sum_{r=1}^{d'} \E[A_{ir}]\gamma_{\pi(i,r)}+\E[b_i] =\gamma_i
\end{align*}
by condition (\ref{cccent}). Hence, the restriction of $T'$ to ${\cal P}$ maps into ${\cal P}$.

Next, we show that the restriction of  $T'$ to ${\cal P}$
is a  contraction with respect to the metric $\ell_2^\vee$: For
$(\mu_1,\ldots,\mu_d), (\nu_1,\ldots,\nu_d)\in {\cal P}$
we first fix $i\in\{1,\ldots,d\}$. Let  $(Z_{ir},Z'_{ir})$ be an optimal coupling of $\mu_{\pi(i,r)}$ and  $\nu_{\pi(i,r)}$ for $r=1,\ldots,d'$ such that   $(Z_{i1},Z'_{i1}),\ldots,(Z_{id'},Z'_{id'}),(A_{i1},\ldots,A_{id'},b_i)$ are independent.
Then we have
\begin{align}\label{darim2}
T'_i(\mu_1,\ldots,\mu_d) = {\cal L}\left(\sum_{r=1}^{d'} A_{ir}Z_{ir} +b_i\right), \quad
T'_i(\nu_1,\ldots,\nu_d) = {\cal L}\left(\sum_{r=1}^{d'} A_{ir}Z'_{ir} +b_i\right).
\end{align}
Denoting by $\overline{\gamma}$ the complex conjugate of $\gamma\in\C$ we obtain
\begin{align}
\lefteqn{\ell_2^2(T'_i(\mu_1,\ldots,\mu_{d}),T'_i(\nu_1,\ldots,\nu_{d}))}\nonumber\\
&\le \E\left[\left|\sum_{r=1}^{d'}  A_{ir}  (Z_{ir}-Z'_{ir})\right|^2\right] \nonumber\\
&= \E\left[ \sum_{r=1}^{d'}  |A_{ir}|^2  |Z_{ir}-Z'_{ir}|^2    \right]+  \E\left[\sum_{r\neq t}  A_{ir} (Z_{ir}-Z'_{ir})   \overline{A_{it} (Z_{it}-Z'_{it})}\right]\nonumber\\
&=\sum_{r=1}^{d'}  \E\left[|A_{ir}|^2\right] \ell_2^2(\mu_{\pi(i,r)}, \nu_{\pi(i,r)}) \label{251212}\\
&\le \left(\sum_{r=1}^{d'}  \E\left[|A_{ir}|^2\right]\right)  \left(\ell_2^\vee((\mu_1,\ldots,\mu_{d}),(\nu_1,\ldots,\nu_{d}))\right)^2.\nonumber
\end{align}
For equality (\ref{251212}) we firstly use that $Z_{ir}-Z'_{ir}$ and $Z_{it}-Z'_{it}$ are independent, centered factors, so that the expectation of the sum over $r\neq t$ is $0$ and secondly that $(Z_{ir},Z'_{ir})$ are optimal couplings of $(\mu_{\pi(i,r)}, \nu_{\pi(i,r)})$ such that $\E[|Z_{ir}-Z'_{ir}|^2]=\ell_2^2(\mu_{\pi(i,r)}, \nu_{\pi(i,r)})$.

Now, taking the maximum over $i$ yields
\begin{align*}
\ell_2^\vee(T'(\mu_1,\ldots,\mu_d),T'(\nu_1,\ldots,\nu_d))
\le \left(\max_{1\le i\le d}\sum_{r=1}^{d'} \E\left[|A_{ir}|^2\right]\right)^{1/2} \ell_2^\vee((\mu_1,\ldots,\mu_d),(\nu_1,\ldots,\nu_d)).
\end{align*}
Hence, condition (\ref{cccont}) implies that the restriction of  $T'$ to ${\cal P}$
is a contraction. Since  the metric $\ell_2^\vee$ is complete, Banach's fixed-point
theorem implies the assertion.
\end{proof}

\section{Convergence and examples}\label{conexp}
In this section a couple of concrete P\'olya urns are considered and convergence of the normalized
numbers of balls of a color is shown within the product metrics defined in section \ref{secspaces}. The proofs are generic such that they can easily be transferred to other urns  of the types (a)--(c) in section \ref{limiteq}. We always show limit laws for the initial compositions of the urn with one ball of (arbitrary) color. Limit laws for other initial compositions can be obtained from these by appropriate convolution with coefficients which are powers of components of an independent Dirichlet distributed vector. We leave the details to the reader.

\subsection{$\mathbf{2 \times 2}$ deterministic replacement urns}  \label{ex1sec}
A discussion of urns with a general balanced $2\times 2$ replacement matrix as in (\ref{rep2x2}) is given in Bagchi and Pal \cite{BagPal85}. Subsequently, we assume the conditions in (\ref{rep2x2}) and, as in \cite{BagPal85}, that  $bc>0$.
As shown in \cite{BagPal85} asymptotic normal behavior occurs for these urns when $a-c\le (a+b)/2$
(type (b) in section \ref{secspaces}), whereas $a-c> (a+b)/2$ leads to limit laws with non-normal limit distributions (type (a) in section \ref{secspaces}). In this section we show how to derive these results by our contraction approach.  With $B_n^\mathrm{b}$ and $B_n^\mathrm{w}$ as in the beginning of section \ref{appr} we denote expectations by $\mu_\mathrm{b}(n)$ and $\mu_\mathrm{w}(n)$.
These values can be derived exactly, see \cite{BagPal85},
\begin{align}
\mu_\mathrm{b}(n)&=
\frac{c(a+b)}{b+c} n
    +\frac{b\,\Gamma\left(\frac{1}{a+b}\right)}
        {(b+c)\Gamma\left(\frac{1+a-c}{a+b}\right)}
      \frac{\Gamma\left(n+\frac{1+a-c}{a+b}\right)}{\Gamma\left(n+\frac1{a+b}\right)}
    +\frac{c}{b+c},   \label{bagpalfor1}\\
\mu_\mathrm{w}(n)&=\frac{c(a+b)}{b+c} n
    -\frac{c\,\Gamma\left(\frac{1}{a+b}\right)}
        {(b+c)\Gamma\left(\frac{1+a-c}{a+b}\right)}
      \frac{\Gamma\left(n+\frac{1+a-c}{a+b}\right)}{\Gamma\left(n+\frac1{a+b}\right)}
    +\frac{c}{b+c}.\label{bagpalfor2}
\end{align}
{\bf Non-normal limit case.}
We first discuss the non-normal case $a-c> (a+b)/2$.  Note that with $\lambda:=(a-c)/(a+b)$ and excluding the case $bc=0$, we have $1/2 < \lambda < 1$ and, as $n\to\infty$,
\begin{align}\label{expexp13}
\mu_\mathrm{b}(n)=
c_\mathrm{b} n
    +d_\mathrm{b} n^\lambda + o(n^\lambda),\quad
\mu_\mathrm{w}(n)=c_\mathrm{w} n
   +d_\mathrm{w} n^\lambda + o(n^\lambda)
\end{align}
with
\begin{align}\label{consts}
 c_\mathrm{b}=c_\mathrm{w}=\frac{c(a+b)}{b+c},\quad    d_\mathrm{b}=  \frac{b\,\Gamma\left(\frac{1}{a+b}\right)}
        {(b+c)\Gamma\left(\frac{1+a-c}{a+b}\right)},\quad
d_\mathrm{w}=-\frac{c\,\Gamma\left(\frac{1}{a+b}\right)}
        {(b+c)\Gamma\left(\frac{1+a-c}{a+b}\right)}.
\end{align}
We use the normalizations $X_0:=Y_0:=0$ and, cf.~\eqref{scala},
\begin{align}\label{normcase11}
X_n:=\frac{B_n^\mathrm{b}-\mu_\mathrm{b}(n)}{n^\lambda},\quad
Y_n:=\frac{B_n^\mathrm{w}-\mu_\mathrm{w}(n)}{n^\lambda}, \quad n\ge 1.
\end{align}
Note that we do not have to identify the order of the variance in advance. It turns out that
 it is sufficient to use the order of the error terms $d_\mathrm{b} n^\lambda$ and $d_\mathrm{w} n^\lambda$ in the expansions (\ref{expexp13}). From the system (\ref{rde1})--(\ref{rde2}) we obtain for the scaled quantities $X_n$, $Y_n$ the system, for $n\ge 1$,
\begin{align}\label{rde1mod}
X_n &\stackrel{d}{=} \sum_{r=1}^{a+1} \left(\frac{I^{(n)}_r}{n}\right)^\lambda X^{(r)}_{I^{(n)}_r} + \sum_{r=a+2}^{K} \left(\frac{I^{(n)}_r}{n}\right)^\lambda Y^{(r)}_{I^{(n)}_r} + b_\mathrm{b}(n), \\
Y_n &\stackrel{d}{=} \sum_{r=1}^{c} \left(\frac{I^{(n)}_r}{n}\right)^\lambda
 X^{(r)}_{I^{(n)}_r} +  \sum_{r=c+1}^{K}
\left(\frac{I^{(n)}_r}{n}\right)^\lambda Y^{(r)}_{I^{(n)}_r} + b_\mathrm{w}(n),\label{rde1modb}
\end{align}
with
 \begin{align}\label{defbbbw}
b_\mathrm{b}(n)&=d_\mathrm{b}\left(-1+\sum_{r=1}^{a+1} \left(\frac{I^{(n)}_r}{n}\right)^\lambda \right)+ d_\mathrm{w}\sum_{r=a+2}^{K} \left(\frac{I^{(n)}_r}{n}\right)^\lambda+o(1),\\
b_\mathrm{w}(n)&=d_\mathrm{b} \sum_{r=1}^{c} \left(\frac{I^{(n)}_r}{n}\right)^\lambda
 + d_\mathrm{w}\left(-1+ \sum_{r=c+1}^{K}  \left(\frac{I^{(n)}_r}{n}\right)^\lambda \right)+o(1),\label{defbbbwb}
\end{align}
with conditions on independence between the $X^{(r)}_j$,$Y^{(r)}_j$ and $I^{(n)}$ and identical distributions of the $X^{(r)}_j$ and $Y^{(r)}_j$ analogously to (\ref{rde1}) and (\ref{rde2}). The $o(1)$ terms in (\ref{defbbbw}) and (\ref{defbbbwb}) are deterministic functions of $I^{(n)}$. In view of Lemma \ref{asysubsize} this suggests for limits $X$ and $Y$ of $X_n$ and $Y_n$ respectively
\begin{align}\label{rde1limit1}
X&\stackrel{d}{=} \sum_{r=1}^{a+1}D_r^\lambda X^{(r)} + \sum_{r=a+2}^{K} D_r^\lambda Y^{(r)} + b_\mathrm{b}, \\
Y &\stackrel{d}{=} \sum_{r=1}^{c}D_r^\lambda
 X^{(r)} +  \sum_{r=c+1}^{K}
D_r^\lambda Y^{(r)} + b_\mathrm{w},\label{rde1limit2}
\end{align}
with
 \begin{align*}
b_\mathrm{b}&=d_\mathrm{b}\left(-1+\sum_{r=1}^{a+1}D_r^\lambda \right)+ d_\mathrm{w}\sum_{r=a+2}^{K} D_r^\lambda,\\
b_\mathrm{w}&=d_\mathrm{b} \sum_{r=1}^{c}D_r^\lambda
 + d_\mathrm{w}\left(-1+ \sum_{r=c+1}^{K}  D_r^\lambda \right),
\end{align*}
where $(D_1,\ldots,D_K)$, $X^{(1)},\ldots,X^{(K)}$, $Y^{(1)},\ldots,Y^{(K)}$ are independent, and the $X^{(r)}$ are distributed as $X$, the $Y^{(r)}$ are distributed as $Y$ and  $(D_1,\ldots,D_K)$ is as in Lemma \ref{asysubsize}.
 Note that the moments $\E[D_r^\lambda]$ and the form of $d_\mathrm{b}$ and $d_\mathrm{w}$ in (\ref{consts}) imply $\E[b_\mathrm{b}]=\E[b_\mathrm{w}]=0$.
 From $\lambda>1/2$ and $\sum_{r=1}^K D_r=1$ we obtain
 \begin{align*}
\sum_{r=1}^{K} \E\left[D_r^{2\lambda}\right] <1.
\end{align*}
Hence, Theorem \ref{fp1} applies to the map associated to the system (\ref{rde1limit1})--(\ref{rde1limit2}) and implies that there exists a unique solution $({\cal L}(\Lambda_\mathrm{b}),{\cal L}(\Lambda_\mathrm{w}))$ in the space ${\cal M}^\R_2(0)\times {\cal M}^\R_2(0)$ to (\ref{rde1limit1})--(\ref{rde1limit2}). The following convergence proof resembles ideas from Neininger and R\"uschendorf \cite{NeRu04}.
\begin{thm}\label{22nnc}
Consider the P\'olya urn with replacement matrix (\ref{rep2x2}) with $a-c> (a+b)/2$ and $bc>0$ and the normalized numbers $X_n$ and $Y_n$  of black balls as in (\ref{normcase11}). Furthermore let  $({\cal L}(\Lambda_\mathrm{b}),{\cal L}(\Lambda_\mathrm{w}))$ denote the in ${\cal M}^\R_2(0)\times {\cal M}^\R_2(0)$ unique solution of (\ref{rde1limit1})--(\ref{rde1limit2}).
Then, as $n\to \infty$,
\begin{align*}
\zeta_2^\vee\left( (X_n,Y_n), (\Lambda_\mathrm{b},\Lambda_\mathrm{w})\right)  \to 0.
\end{align*}
In particular,  as $n\to \infty$,
\begin{align} \label{limperib}
X_n \stackrel{d}{\longrightarrow} \Lambda_\mathrm{b}, \quad  Y_n \stackrel{d}{\longrightarrow} \Lambda_\mathrm{w}.
\end{align}
\end{thm}
\begin{proof}
We first define, for $n\ge 1$, the accompanying sequences
\begin{align} \label{limperid}
Q^\mathrm{b}_n &:=  \sum_{r=1}^{a+1} \left(\frac{I^{(n)}_r}{n}\right)^\lambda \Lambda^{(r)}_\mathrm{b} + \sum_{r=a+2}^{K} \left(\frac{I^{(n)}_r}{n}\right)^\lambda \Lambda^{(r)}_\mathrm{w} + b_\mathrm{b}(n), \\
Q^\mathrm{w}_n &:=\sum_{r=1}^{c} \left(\frac{I^{(n)}_r}{n}\right)^\lambda
 \Lambda^{(r)}_\mathrm{b} +  \sum_{r=c+1}^{K}
\left(\frac{I^{(n)}_r}{n}\right)^\lambda \Lambda^{(r)}_\mathrm{w} + b_\mathrm{w}(n),
\end{align}
with $b_\mathrm{b}(n)$ and $b_\mathrm{w}(n)$ as in (\ref{defbbbw}) and  the $\Lambda^{(r)}_\mathrm{b}$,   $\Lambda^{(r)}_\mathrm{b}$ and  $I^{(n)}$ being independent, where the $\Lambda^{(r)}_\mathrm{b}$ are distributed as $\Lambda_\mathrm{b}$ and the
$\Lambda^{(r)}_\mathrm{w}$ are distributed as $\Lambda_\mathrm{w}$ for the respective values of $r$. Note that $Q^\mathrm{b}_n$ and $Q^\mathrm{w}_n$ are centered with finite second moment since ${\cal L}(\Lambda_\mathrm{b}),{\cal L}(\Lambda_\mathrm{b}) \in {\cal M}^\R_2(0)$. Hence, $\zeta_2$ distances between $X_n, Y_n, Q^\mathrm{b}_n, Q^\mathrm{w}_n, \Lambda_\mathrm{b}$ and $\Lambda_\mathrm{w}$ are finite. To bound
$$\Delta(n):=\zeta_2^\vee((X_n,Y_n), (\Lambda_\mathrm{b},\Lambda_\mathrm{w}))$$
we look at the distances
$$\Delta_\mathrm{b}(n):=\zeta_2(X_n, \Lambda_\mathrm{b}), \quad \Delta_\mathrm{w}(n):=\zeta_2(Y_n, \Lambda_\mathrm{w}).$$
We start with the estimate
\begin{align} \label{triest}
\zeta_2(X_n, \Lambda_\mathrm{b}) \le \zeta_2(X_n, Q_n^\mathrm{b}) + \zeta_2(Q_n^\mathrm{b}, \Lambda_\mathrm{b}).
\end{align}
We first show for the second summand in the latter display that $\zeta_2(Q_n^\mathrm{b}, \Lambda_\mathrm{b}) \to 0$ as $n\to\infty$: With  inequality (\ref{zetaellest}) we have
\begin{align*}
 \zeta_2(Q_n^\mathrm{b}, \Lambda_\mathrm{b})\le (\|Q_n^\mathrm{b}\|_2 + \|\Lambda_\mathrm{b}\|_2) \ell_2(Q_n^\mathrm{b}, \Lambda_\mathrm{b}).
\end{align*}
Moreover, $\|\Lambda_\mathrm{b}\|_2<\infty$ since ${\cal L}(\Lambda_\mathrm{b})\in {\cal M}^\R_2$ and, by definition of $Q_n^\mathrm{b}$ and with $|I^{(n)}_r/n|\le 1$ we have that $\|Q_n^\mathrm{b}\|_2$ is uniformly bounded in $n$. Hence, it is sufficient to show $\ell_2(Q_n^\mathrm{b}, \Lambda_\mathrm{b}) \to 0$. We have, using the independence properties in (\ref{limperid}) and (\ref{rde1limit1}), that
\begin{align*}
\lefteqn{\ell_2(Q_n^\mathrm{b}, \Lambda_\mathrm{b})}\\
& \le \sum_{r=1}^{a+1} \left\| \left(\frac{I^{(n)}_r}{n}\right)^\lambda - D_r^\lambda \right\|_2 \|\Lambda^{(r)}_\mathrm{b}\|_2 + \sum_{r=a+2}^{K}
 \left\| \left(\frac{I^{(n)}_r}{n}\right)^\lambda - D_r^\lambda \right\|_2 \|\Lambda^{(r)}_\mathrm{w}\|_2 + \|b_\mathrm{b}(n) - b_\mathrm{b}\|_2.
\end{align*}
Lemma \ref{asysubsize} implies that $\|(I^{(n)}_r/n)^\lambda - D_r^\lambda\|_2 \to 0$ as $n\to\infty$, which as well implies $\|b_\mathrm{b}(n) - b_\mathrm{b}\|_2 \to 0$. Hence,  we obtain $\ell_2(Q_n^\mathrm{b}, \Lambda_\mathrm{b}) \to 0$ and $\zeta_2(Q_n^\mathrm{b}, \Lambda_\mathrm{b}) \to 0$.

Next, we bound the first summand  $\zeta_2(X_n, Q_n^\mathrm{b})$ in (\ref{triest}). We condition on $I^{(n)}$. Note that conditionally on $I^{(n)}$ we have that $b_\mathrm{b}(n)$ is deterministic, which, for integration, we denote by $\beta=\beta(I^{(n)})$. Denoting the distribution of $I^{(n)}$ by $\Upsilon_n$ and $\mathbf{i}:=(i_1,\ldots,i_K)$ this yields
\begin{align}
\zeta_2(X_n, Q_n^\mathrm{b}) &\le \int \zeta_2\left( \sum_{r=1}^{a+1} \left(\frac{i_r}{n}\right)^\lambda X^{(r)}_{i_r} + \sum_{r=a+2}^{K} \left(\frac{i_r}{n}\right)^\lambda Y^{(r)}_{i_r} +\beta,    \right.  \nonumber \\
&\quad \quad\quad \quad \left. \sum_{r=1}^{a+1} \left(\frac{i_r}{n}\right)^\lambda \Lambda^{(r)}_\mathrm{b} + \sum_{r=a+2}^{K} \left(\frac{i_r}{n}\right)^\lambda \Lambda^{(r)}_\mathrm{w} + \beta\right) \,d\Upsilon_n(\mathbf{i})\nonumber\\
&\le  \int \left(\sum_{r=1}^{a+1} \left(\frac{i_r}{n}\right)^{2\lambda} \zeta_2(X^{(r)}_{i_r},\Lambda^{(r)}_\mathrm{b}) + \sum_{r=a+2}^{K} \left(\frac{i_r}{n}\right)^{2\lambda} \zeta_2(Y^{(r)}_{i_r},\Lambda^{(r)}_\mathrm{w})\right) d\Upsilon_n(\mathbf{i})  \label{81131}\\
&=\sum_{r=1}^{a+1} \E\left[\left(\frac{I^{(n)}_r}{n}\right)^{2\lambda}\Delta_\mathrm{b}(I^{(n)}_r)\right] + \sum_{r=a+2}^{K} \E\left[\left(\frac{I^{(n)}_r}{n}\right)^{2\lambda}\Delta_\mathrm{w}(I^{(n)}_r)\right]\nonumber\\
&\le \sum_{r=1}^{K}\E\left[\left(\frac{I^{(n)}_r}{n}\right)^{2\lambda}\Delta(I^{(n)}_r)\right],\nonumber
\end{align}
where, for (\ref{81131}) we use that $\zeta_2$ is $(2,+)$ ideal, as well as (\ref{sumabsch}).
Altogether, the estimate started in (\ref{triest}) yields
\begin{align*}
\Delta_\mathrm{b}(n) \le  \sum_{r=1}^{K}\E\left[\left(\frac{I^{(n)}_r}{n}\right)^{2\lambda}\Delta(I^{(n)}_r)\right]+ o(1).
\end{align*}
With the same argument we obtain the same upper bound for $\Delta_\mathrm{w}(n)$. Thus, using also that $I^{(n)}_1,\ldots,I^{(n)}_K$ are identically distributed we have
\begin{align}\label{rec_fia}
\Delta(n) \le K\E\left[\left(\frac{I^{(n)}_1}{n}\right)^{2\lambda}\Delta(I^{(n)}_1)\right] + o(1).
\end{align}
Now, a standard argument implies $\Delta(n)\to 0$ as follows: First from (\ref{rec_fia}) we obtain with $I^{(n)}_1/n\to D_1$ in $L_2$ and, by $\lambda>1/2$, with $\vartheta:= K \E[D_1^{2\lambda}]<1$  that
\begin{align*}
\Delta(n) &\le  K\E\left[\left(\frac{I^{(n)}_1}{n}\right)^{2\lambda}\right]\max_{0\le k\le n-1} \Delta(k) + o(1)\\
&\le  (\vartheta+o(1))\max_{0\le k\le n-1} \Delta(k)+o(1).
\end{align*}
Since $\vartheta<1$ this implies that the sequence $(\Delta(n))_{n\ge 0}$ is bounded. We denote $\eta:=\sup_{n\ge 0} \Delta(n)$ and $\xi:=\limsup_{n\to\infty} \Delta(n)$. For any $\varepsilon> 0$ there exists an $n_0 \ge 0$ such that $\Delta(n) \le \xi+\varepsilon$  for all
$n \ge n_0$. Hence, from (\ref{rec_fia}) we obtain
\begin{align*}
\Delta(n) \le K\E\left[{\bf 1}_{\left\{I^{(n)}_1<n_0\right\}} \left(\frac{I^{(n)}_1}{n}\right)^{2\lambda} \right]\eta+ K\E\left[{\bf 1}_{\left\{I^{(n)}_1\ge n_0\right\}} \left(\frac{I^{(n)}_1}{n}\right)^{2\lambda} \right](\xi+\varepsilon) + o(1).
\end{align*}
With $n\to\infty$ this implies
\begin{align*}
\xi \le  \vartheta(\xi+\varepsilon).
\end{align*}
Since $ \vartheta<1$ and $\varepsilon>0$ is arbitrary this implies $\xi=0$. Hence, we have
$\zeta_2^\vee\left( (X_n,Y_n), (\Lambda_\mathrm{b},\Lambda_\mathrm{w})\right)  \to 0$ as $n\to \infty$. Since convergence in $\zeta_2$ implies weak convergence, this implies (\ref{limperib}) as well.
\end{proof}

\noindent
{\bf The normal limit case.} Now, we discuss the normal limit case $a-c \le (a+b)/2$, where we first consider $a-c < (a+b)/2$. (The remaining case $a-c =(a+b)/2$ is similar with more involved expansions for the first two moments.) The formulae (\ref{bagpalfor1}), (\ref{bagpalfor2}) now imply
 \begin{align}\label{expexp13b}
\mu_\mathrm{b}(n)=
c_\mathrm{b} n
    + o(\sqrt{n}),\quad
\mu_\mathrm{w}(n)=c_\mathrm{w} n
    + o(\sqrt{n})
\end{align}
with $c_\mathrm{b}$ and $c_\mathrm{w}$ as in (\ref{consts}).
As it is usual in the use of the contraction method for proving normal limit laws based on the metric $\zeta_3$ we need also an expansion of the variance. We denote the variances of
 $B_n^\mathrm{b}$ and $B_n^\mathrm{w}$ by  $\sigma^2_\mathrm{b}(n)$ and $\sigma^2_\mathrm{w}(n)$. Additionally to   $bc=0$ we  exclude the case $a=c$. (In this case there is a trivial non-random evolution of the urn).  From \cite{BagPal85} we have as $n\to\infty$:
\begin{align}\label{async22}
\sigma^2_\mathrm{b}(n) = f_\mathrm{b} n + o(n), \quad \sigma^2_\mathrm{w}(n) = f_\mathrm{w} n + o(n),
\end{align}
with
\begin{align*}
f_\mathrm{b}=f_\mathrm{w}=
\frac{(a+b)bc(a-c)^2}{(a+b-2(a-c))(b+c)^2}>0.
\end{align*}
We use the normalizations $X_0:=Y_0:=X_1:=Y_1:=0$ and, cf.~\eqref{scalb},
\begin{align}\label{normcase11b}
X_n:=\frac{B_n^\mathrm{b}-\mu_\mathrm{b}(n)}{\sigma_\mathrm{b}(n)},\quad
Y_n:=\frac{B_n^\mathrm{w}-\mu_\mathrm{w}(n)}{\sigma_\mathrm{w}(n)}, \quad n\ge 2.
\end{align}
From the system (\ref{rde1})--(\ref{rde2}) we obtain for the scaled
quantities $X_n$, $Y_n$ the system, for $n\ge 1$,
\begin{align}\label{rde1modnew}
  X_n &\stackrel{d}{=} \sum_{r=1}^{a+1}
  \frac{\sigma_\mathrm{b}(I^{(n)}_r)}{\sigma_\mathrm{b}(n)}
  X^{(r)}_{I^{(n)}_r} + \sum_{r=a+2}^{K}
  \frac{\sigma_\mathrm{w}(I^{(n)}_r)}{\sigma_\mathrm{b}(n)}
  Y^{(r)}_{I^{(n)}_r} + e_\mathrm{b}(n),
  \\
  Y_n &\stackrel{d}{=} \sum_{r=1}^{c}
  \frac{\sigma_\mathrm{b}(I^{(n)}_r)}{\sigma_\mathrm{w}(n)}
  X^{(r)}_{I^{(n)}_r} + \sum_{r=c+1}^{K}
  \frac{\sigma_\mathrm{w}(I^{(n)}_r)}{\sigma_\mathrm{w}(n)}
  Y^{(r)}_{I^{(n)}_r} + e_\mathrm{w}(n),\label{rde1modnewb}
\end{align}
with conditions on independence and identical distributions analogously to (\ref{rde1}) and (\ref{rde2}) respectively (\ref{rde1mod}) and (\ref{rde1modb}). We have  $\|e_\mathrm{b}(n)\|_\infty, \|e_\mathrm{w}(n)\|_\infty\to 0$ since the leading linear terms in the expansions (\ref{expexp13b}) cancel out and the error terms $o(\sqrt{n})$ are asymptotically eliminated by the scaling of order $1/\sqrt{n}$.
 In view of Lemma \ref{asysubsize} this suggests for limits $X$ and $Y$ of $X_n$ and $Y_n$ respectively
\begin{align}\label{rde1limit1b}
X&\stackrel{d}{=} \sum_{r=1}^{a+1} \sqrt{D_r} X^{(r)} + \sum_{r=a+2}^{K} \sqrt{D_r} Y^{(r)}, \\
Y &\stackrel{d}{=} \sum_{r=1}^{c}\sqrt{D_r}
 X^{(r)} +  \sum_{r=c+1}^{K}
\sqrt{D_r} Y^{(r)},\label{rde1limit2b}
\end{align}
where $(D_1,\ldots,D_K)$, $X^{(1)},\ldots,X^{(K)}$, $Y^{(1)},\ldots,Y^{(K)}$ are independent, and the $X^{(r)}$ are distributed as $X$ and the $Y^{(r)}$ are distributed as $Y$.
 To the map associated to the system (\ref{rde1limit1b})--(\ref{rde1limit2b}) we can apply Theorem \ref{fp2}. The conditions (\ref{normcoe}) and (\ref{techcon}) are trivially satisfied. Hence $({\cal N}(0,1),{\cal N}(0,1))$ is the unique fixed-point of the associated map  in the space ${\cal M}^\R_3(0,1)\times {\cal M}^\R_3(0,1)$.
\begin{thm}\label{lim22nor}
Consider the P\'olya urn with replacement matrix (\ref{rep2x2}) with $a-c< (a+b)/2$ and $bc>0$ and the normalized numbers $X_n$ and $Y_n$  of black balls as in (\ref{normcase11b}).
Then, as $n\to \infty$,
\begin{align*}
\zeta_3^\vee\left( (X_n,Y_n), ({\cal N}(0,1),{\cal N}(0,1))\right)  \to 0.
\end{align*}
In particular,  as $n\to \infty$,
\begin{align*}
X_n \stackrel{d}{\longrightarrow} {\cal N}(0,1), \quad  Y_n \stackrel{d}{\longrightarrow} {\cal N}(0,1).
\end{align*}
\end{thm}
\begin{proof}
The proof of this Theorem can be done along the lines of the proof of Theorem \ref{22nnc}. However, more care has to be taken in the definition of the quantities corresponding to
$Q^\mathrm{b}_n$ and $Q^\mathrm{w}_n$ in (\ref{limperid}) in order to assure finiteness of the $\zeta_3$ distances. A possible choice is, for $n\ge 2$,
\begin{align} \label{limperi2}
  \tilde{Q}^\mathrm{b}_n &:=
  \sum_{r=1}^{a+1}{\bf 1}_{\left\{I^{(n)}_r\ge 2\right\}}
  \frac{\sigma_\mathrm{b}(I^{(n)}_r)}{\sigma_\mathrm{b}(n)} N_r
  + \sum_{r=a+2}^{K}{\bf 1}_{\left\{I^{(n)}_r\ge 2\right\}}
  \frac{\sigma_\mathrm{w}(I^{(n)}_r)}{\sigma_\mathrm{b}(n)} N_r
  + e_\mathrm{b}(n),
  \\
  \tilde{Q}^\mathrm{w}_n&\stackrel{d}{=}
  \sum_{r=1}^{c}{\bf 1}_{\left\{I^{(n)}_r\ge 2\right\}}
  \frac{\sigma_\mathrm{b}(I^{(n)}_r)}{\sigma_\mathrm{w}(n)} N_r
  + \sum_{r=c+1}^{K} {\bf 1}_{\left\{I^{(n)}_r\ge 2\right\}}
  \frac{\sigma_\mathrm{w}(I^{(n)}_r)}{\sigma_\mathrm{w}(n)} N_r
  +  e_\mathrm{w}(n),
\end{align}
with $e_\mathrm{b}(n)$ and $e_\mathrm{w}(n)$ as in (\ref{rde1modnew})--(\ref{rde1modnewb}) and  $N_1,\ldots,N_K$, $I^{(n)},$  independent, where the $N_r$ are standard normally
distributed for $r=1,\ldots,K$. A comparison on the definition of $\tilde{Q}^\mathrm{b}_n$ and
$\tilde{Q}^\mathrm{w}_n$ with the right hand sides of (\ref{rde1modnew}) and  (\ref{rde1modnewb}) and the scaling (\ref{normcase11b}) yields that we have
$\E[\tilde{Q}^\mathrm{b}_n]=\E[\tilde{Q}^\mathrm{w}_n]=0$ and $\Var(\tilde{Q}^\mathrm{b}_n)=\Var(\tilde{Q}^\mathrm{w}_n)=1$ for all $n\ge 2$. Obviously, we also have $\|\tilde{Q}^\mathrm{b}_n\|_3, \|\tilde{Q}^\mathrm{w}_n\|_3<\infty$.
Hence, $\zeta_3$ distances between $X_n$, $Y_n$, $\tilde{Q}^\mathrm{b}_n$, $\tilde{Q}^\mathrm{w}_n$, and ${\cal N}(0,1)$ are finite for all $n\ge 2$. With
\begin{align*}
\tilde{\Delta}(n)&:=\zeta_3^\vee((X_n,Y_n), ({\cal N}(0,1),{\cal N}(0,1))), \\
\tilde{\Delta}_\mathrm{b}(n)&:=\zeta_3(X_n, {\cal N}(0,1)), \\
\tilde{\Delta}_\mathrm{w}(n)&:=\zeta_3(Y_n, {\cal N}(0,1))
\end{align*}
we also start with
\begin{align*}
\zeta_3(X_n, {\cal N}(0,1)) \le \zeta_3(X_n, \tilde{Q}_n^\mathrm{b}) + \zeta_3(\tilde{Q}_n^\mathrm{b}, {\cal N}(0,1)).
\end{align*}
Analogously to the proof of Theorem \ref{22nnc} we obtain  $\zeta_3(\tilde{Q}_n^\mathrm{b}, {\cal N}(0,1))\to 0$ as $n\to\infty$.

The bound for  $\zeta_3(X_n, \tilde{Q}_n^\mathrm{b})$ is also analogous to the proof of Theorem \ref{22nnc}, where now is used that $\zeta_3$ is $(3,+)$ ideal instead of $(2,+)$ ideal. This yields
\begin{align*}
  \zeta_3(X_n, \tilde{Q}_n^\mathrm{b})
  &\le \sum_{r=1}^{a+1}\E\left[
    \left(\frac{\sigma_\mathrm{b}(I^{(n)}_r)}{\sigma_\mathrm{b}(n)}\right)^{\!\!3}
    \tilde{\Delta}(I^{(n)}_r)\right]
  + \sum_{r=a+2}^{K}\E\left[
    \left(\frac{\sigma_\mathrm{w}(I^{(n)}_r)}{\sigma_\mathrm{b}(n)}\right)^{\!\!3}
    \tilde{\Delta}(I^{(n)}_r)\right].
\end{align*}
Then we argue as in the previous proof to obtain analogous to (\ref{rec_fia})
\begin{align*}
  \tilde{\Delta}(n)\le
 \sum_{r=1}^{a+1}\E\left[
   \left(\frac{\sigma_\mathrm{b}(I^{(n)}_r)}{\sigma_\mathrm{b}(n)}\right)^{\!\!3}
   \tilde{\Delta}\bigl(I^{(n)}_r\bigr)\right]
 + \sum_{r=a+2}^{K}\E\left[
   \left(\frac{\sigma_\mathrm{w}(I^{(n)}_r)}{\sigma_\mathrm{b}(n)}\right)^{\!\!3}
   \tilde{\Delta}\bigl(I^{(n)}_r\bigr)\right]
 +o(1).
\end{align*}
From this estimate we can deduce $\tilde{\Delta}(n) \to 0$ as for $\Delta(n)$ in the proof of Theorem \ref{22nnc}, where we need to use that from the expansions (\ref{async22}) and Lemma \ref{asysubsize} we obtain, as $n\to\infty$, that
\begin{align}\label{8113y}
  \sum_{r=1}^{a+1}\E\left[\left(
      \frac{\sigma_\mathrm{b}\bigl(I^{(n)}_r\bigr)}{\sigma_\mathrm{b}(n)}
    \right)^{\!\!3}\right]
  + \sum_{r=a+2}^{K}\E\left[\left(
      \frac{\sigma_\mathrm{w}\bigl(I^{(n)}_r\bigr)}{\sigma_\mathrm{b}(n)}
    \right)^{3}\right]
  \to \sum_{r=1}^{K} \E\left[D_r^{3/2}\right] <1.
\end{align}
\end{proof}
\noindent
{\bf Remarks.} {\bf (1)} Note that the proof of Theorem \ref{lim22nor} cannot be done in the $\zeta_2^\vee$ metric since the term corresponding to (\ref{8113y}) then is
 \begin{align*}
\sum_{r=1}^{a+1}\E\left[\left(\frac{\sigma_\mathrm{b}(I^{(n)}_r)}{\sigma_\mathrm{b}(n)}\right)^{2}\right] + \sum_{r=a+2}^{K}\E\left[\left(\frac{\sigma_\mathrm{w}(I^{(n)}_r)}{\sigma_\mathrm{b}(n)}\right)^{2}\right] \to \sum_{r=1}^{K} \E\left[D_r\right] =1,
\end{align*}
where a limit $<1$ is required to obtain $\tilde{\Delta}(n) \to 0$. This is the reason, why we use $\zeta_3^\vee$. It is possible to use $\zeta_s^\vee$ for any $2<s\le 3$ leading to the limit $\sum_{r=1}^K \E[D_r^s] <1$.\\
{\bf (2)} The case $a-c =(a+b)/2$ differs in the error terms in (\ref{expexp13b}) which then become $\bo(\sqrt{n})$. Since the variances in (\ref{async22}) get additional logarithmic factors we still obtain the system (\ref{rde1limit1b})--(\ref{rde1limit2b}) and our proof technique can be applied as well.\\
{\bf (3)} The condition $bc>0$ cannot be dropped. In the case $bc=0$ the urn  model is not irreducible in the terminology of Janson \cite{Ja04} and is known to behave quite differently. A  comprehensive study of the case $bc=0$ is given in Janson \cite{Ja06T}, see also Janson \cite{Ja10}.  In our approach $bc=0$
would  lead to degenerate systems of limit equations that do not
identify  limit laws.\\
{\bf (4)} The condition $f_{b}=f_{w}$ is necessary for our  proof to work.

\subsection{An urn with random replacements}\label{kersting}
As an example for  random entries in the replacement matrix $R$ we consider a simple model with two colors, black and white. In each step when drawing a black ball, a coin is independently tossed to decide whether the black ball is placed back together with another black ball or together with another white ball. The probability for success (a second black ball) is denoted by $0<\alpha<1$.
Similarly, if a white ball is drawn a coin with probability $0<\beta<1$ is tossed to decide whether a second white ball or a black ball is placed back together with the white ball. We denote the replacement matrix by
\begin{align}\label{reprand2x2}
  R=\left[ \begin{array}{cc} F_{\alpha} & 1-F_{\alpha} \\ 1-F_{\beta} & F_{\beta}\end{array} \right]
\end{align}
where $F_\alpha$ and $F_\beta$ denote Bernoulli random variables being $1$ with probabilities $\alpha$ and $\beta$ respectively, otherwise $0$. This urn model has been introduced in the context of clinical trials and been studied together with generalizations in \cite{wedu78,wesmkiya90,smro95,sm96,maro97,bahu99,bahuzh02,Ja04}.

The row sums of $R$ in (\ref{reprand2x2}) are both almost surely equal to one, hence the urn is balanced.
Again, the number of black balls after $n$ draws starting with an initial composition with one black ball is denoted by $B^\mathrm{b}_n$, when starting with a white ball by $B^\mathrm{w}_n$.
According to our approach in section \ref{appr} we obtain the recursive equation
\begin{align}\label{rder1}
  B_n^{\mathrm{b}} \stackrel{d}{=} B^{\mathrm{b}, (1)}_{I_n} + F_{\alpha}B^{\mathrm{b}, (2)}_{J_n} + \left(1-F_{\alpha}\right) B^{\mathrm{w}}_{J_n}, \quad n\ge 1,
\end{align}
where $(B^{\mathrm{b}, (1)}_k)_{0\le k < n}, (B^{\mathrm{b},
  (2)}_k)_{0\le k < n}$,  $(B^{\mathrm{w}}_k)_{0\le k < n}$, $F_\alpha$ and
$I_n$ are independent, and $B^{\mathrm{b}, (1)}_k$ and
$B^{\mathrm{b}, (2)}_k$ are distributed as $B^{\mathrm{b}}_k$ for
$k=0,\ldots,n-1$, and $I_n$ is uniformly distributed on $\{0,\dots,n-1\}$ while $J_n:=n-1-I_n$. (The uniform distribution of $I_n$ follows from the uniform distribution of the number of balls in the $\left[\begin{smallmatrix}1 & 0 \\ 0 &1\end{smallmatrix} \right]$-P\'olya urn.)
Similarly, we obtain for $B_n^\mathrm{w}$ that
\begin{align}\label{rder2}
B_n^\mathrm{w} \stackrel{d}{=}  B^{\mathrm{w}, (1)}_{I_n} +F_{\beta}  B^{\mathrm{w}, (2)}_{J_n} + \left(1-F_{\beta}\right) B^{\mathrm{b}}_{J_n} , \quad n\ge 1,
\end{align}
with conditions on independence and identical distributions similar to (\ref{rder1}). Together with the initial value $(B_0^\mathrm{b},B_0^\mathrm{w})=(1,0)$ the system of equations (\ref{rder1})--(\ref{rder2}) again defines the sequence of pairs of distributions $({\cal L}(B_n^\mathrm{b}),{\cal L}(B_n^\mathrm{w}))_{n\ge 0}$.
As a special case of Lemma \ref{asysubsize} we have
\begin{align}\label{coefcon}
(I_n,J_n)\to (U,1-U) \quad, (n\to\infty),
\end{align}
almost surely  where $U$ is uniformly distributed on $[0,1]$.
Furthermore, we denote for $n\ge 0$
\begin{align}\label{meanclinic}
\mu_\mathrm{b}(n):=\E [ B^\mathrm{b}_n  ], \quad
\mu_{\mathrm{w}}(n):=
\E [ B^\mathrm{w}_n ].
\end{align}
These means have been studied before. We have the following exact formulae:
\begin{lem}\label{ewker}
For $\mu_\mathrm{b}(n)$ and $\mu_\mathrm{w}(n)$ as in (\ref{meanclinic}) with $0<\alpha,\beta<1$ we have
\begin{align}
  \label{erwrandb}
  \mu_{\mathrm{b}}\left(n\right)
  &=\frac{1-\beta}{2-\alpha-\beta}\, n
  +\frac{1-\alpha}{2-\alpha-\beta}
      \frac{\Gamma(n+\alpha+\beta)}{\Gamma(\alpha+\beta)\Gamma(n+1)}
  +\frac{1-\beta}{2-\alpha-\beta},\\
  \label{erwrandw}
\mu_{\mathrm{w}}\left(n\right)
&=\frac{1-\beta}{2-\alpha-\beta}\, n
  -\frac{1-\beta}{2-\alpha-\beta}
      \frac{\Gamma(n+\alpha+\beta)}{\Gamma(\alpha+\beta)\Gamma(n+1)}
  +\frac{1-\beta}{2-\alpha-\beta}.
\end{align}
\end{lem}
\begin{proof}
A proof is based on  matrix diagonalization and can  elementary be  done along the lines of the proof of Lemma \ref{lemewcyc} below.
\end{proof}
As in the example from section \ref{ex1sec} we have two different types of limit laws, with normal limit for $\alpha+\beta\le 3/2$ and with non-normal limit for $\alpha+\beta>3/2$.\\

\noindent
{\bf The non-normal limit case.} We  assume that $\lambda:=\alpha+\beta-1>1/2$.
From Lemma \ref{ewker}  we obtain the asymptotic expressions, as $n\to\infty$,
\begin{align*}
\mu_\mathrm{b}(n)&=
c'_\mathrm{b} n
    +d'_\mathrm{b} n^\lambda + o(n^\lambda),\\
\mu_\mathrm{w}(n)&=c'_\mathrm{w} n
   +d'_\mathrm{w} n^\lambda + o(n^\lambda)
\end{align*}
with  constants
\begin{align}\label{randcon}
 c'_\mathrm{b}=c'_\mathrm{w}=\frac{1-\beta}{1-\lambda},\quad
d'_\mathrm{b}=  \frac{1-\alpha}{\left(1-\lambda\right)\Gamma(\lambda+1)},\quad
d'_\mathrm{w}=-\frac{1-\beta}{\left(1-\lambda\right)\Gamma(\lambda+1)}.
\end{align}
We use the normalizations $X_0:=Y_0:=0$ and, cf.~\eqref{scala},
\begin{align}
 \label{eq:2}
 X_{n}:=\frac{B_{n}^\mathrm{b}-\mu_\mathrm{b}(n)}{n^\lambda},
 \quad
 Y_{n}:=\frac{B_{n}^\mathrm{w}-\mu_\mathrm{w}(n)}{n^\lambda},
 \quad n\ge 1.
\end{align}
As in the non-normal case of the example in section \ref{ex1sec} it is sufficient to use the order of the error term of the mean for the scaling. From  (\ref{rder1})--(\ref{rder2}) we obtain for $n\ge 1$,
\begin{align}\label{rderscal}
  X_n &\stackrel{d}{=} \left(\frac{I_n}{n}\right)^\lambda X^{(1)}_{I_n}
+ F_{\alpha}\left(\frac{J_n}{n}\right)^\lambda X^{(2)}_{J_n}
+ \left(1-F_{\alpha}\right) \left(\frac{J_n}{n}\right)^\lambda Y_{J_n} + b'_\mathrm{b}(n), \\
  Y_n &\stackrel{d}{=}   \left(\frac{I_n}{n}\right)^\lambda Y^{(1)}_{I_n} +F_{\beta}\left(\frac{J_n}{n}\right)^\lambda  Y^{(2)}_{J_n} + \left(1-F_{\beta}\right)\left(\frac{J_n}{n}\right)^\lambda X_{J_n}+ b'_\mathrm{w}(n),
\end{align}
with
\begin{align*}
b'_\mathrm{b}(n)  &= d'_\mathrm{b} \left(\left(\frac{I_n}{n}\right)^\lambda + F_\alpha \left(\frac{J_n}{n}\right)^\lambda -1 \right) +d'_\mathrm{w}(1-F_\alpha)  \left(\frac{J_n}{n}\right)^\lambda+o(1),\\
b'_\mathrm{w}(n)  &= d'_\mathrm{w} \left(\left(\frac{I_n}{n}\right)^\lambda + F_\beta \left(\frac{J_n}{n}\right)^\lambda -1 \right) +d'_\mathrm{b}(1-F_\beta)  \left(\frac{J_n}{n}\right)^\lambda+o(1),
\end{align*}
with conditions on independence and identical distributions analogously to (\ref{rder1})--(\ref{rder2}). In view of (\ref{coefcon}) this suggests for limits $X$ and $Y$ of $X_n$ and $Y_n$ that
\begin{align}\label{limeq12}
  X &\stackrel{d}{=} U^\lambda X^{(1)}
+ F_\alpha (1-U)^\lambda X^{(2)}
+ \left(1-F_\alpha\right) (1-U)^\lambda Y^{(1)} + b'_\mathrm{b}, \\
  Y &\stackrel{d}{=}  U^\lambda Y^{(1)} +F_\beta(1-U)^\lambda  Y^{(2)} + \left(1-F_\beta \right)(1-U)^\lambda X^{(1)}+ b'_\mathrm{w}, \label{limeq12b}
\end{align}
with
\begin{align*}
b'_\mathrm{b}  &= d'_\mathrm{b} \left(U^\lambda + F_\alpha (1-U)^\lambda -1 \right) +d'_\mathrm{w}(1-F_\alpha)  (1-U)^\lambda,\\
b'_\mathrm{w}  &= d'_\mathrm{w} \left(U^\lambda + F_\beta (1-U)^\lambda -1 \right) +d'_\mathrm{b}(1-F_\beta)  (1-U)^\lambda,
\end{align*}
where $X^{(1)}$, $X^{(2)}$, $Y^{(1)}$, $Y^{(2)}$ and $U$ are independent and $X^{(1)}$, $X^{(2)}$ are distributed as $X$ and $Y^{(1)}$, $Y^{(2)}$ are distributed as $Y$.

To check that Theorem \ref{fp1} can be applied to the map associated to the system
(\ref{limeq12})--(\ref{limeq12b}) first note  that the  form of $d'_\mathrm{b}$ and $d'_\mathrm{w}$ in (\ref{randcon}) implies $\E[b'_\mathrm{b}]=\E[b'_\mathrm{w}]=0$. To check condition (\ref{contcond1}) note that we have
\begin{align*}
 \E\left[U^{2\lambda}\right]
+ \E\left[F_\alpha (1-U)^{2\lambda}\right]
+ \E\left[(1-F_\alpha) (1-U)^{2\lambda}\right] =\frac{2}{2\lambda+1}<1,
\end{align*}
since $\lambda>1/2$.
Analogously, we have $\E[U^{2\lambda}]
+ \E[F_\beta (1-U)^{2\lambda}]
+ \E[(1-F_\beta) (1-U)^{2\lambda}] =2/(2\lambda+1)<1$. Together, this verifies condition (\ref{contcond1}). Hence Theorem \ref{fp1} can be applied and yields a unique fixed-point $({\cal L}(\Lambda'_\mathrm{b}), {\cal L}(\Lambda'_\mathrm{w}))$ in ${\cal M}^\R_2(0)\times {\cal M}^\R_2(0)$ to (\ref{limeq12})--(\ref{limeq12b}).
\begin{thm}\label{22ran}
Consider the P\'olya urn with random replacement matrix (\ref{reprand2x2}) with $\alpha,\beta\in(0,1)$ and $\alpha+\beta > 3/2$ and the normalized numbers $X_n$ and $Y_n$  of black balls as in (\ref{eq:2}). Furthermore let  $({\cal L}(\Lambda'_\mathrm{b}),{\cal L}(\Lambda'_\mathrm{w}))$ denote the in ${\cal M}^\R_2(0)\times {\cal M}^\R_2(0)$ unique solution of (\ref{limeq12})--(\ref{limeq12b}).
Then, as $n\to \infty$,
\begin{align*}
X_n \stackrel{d}{\longrightarrow} \Lambda'_\mathrm{b}, \quad  Y_n \stackrel{d}{\longrightarrow} \Lambda'_\mathrm{w}.
\end{align*}
\end{thm}
\begin{proof} Analogously to the proof of Theorem \ref{22nnc}. \end{proof}

\noindent
{\bf The normal limit case.} Now, we discuss the normal limit case  $\lambda:=\alpha+\beta-1\le 1/2$. We first assume  $\lambda:=\alpha+\beta-1< 1/2$.  The  expansions from Lemma \ref{ewker} now imply, as $n\to\infty$
\begin{align}\label{expmeanrannorm}
\mu_\mathrm{b}(n) = c_\mathrm{b} n + o(\sqrt{n}),\quad \mu_\mathrm{w}(n) = c_\mathrm{w} n + o(\sqrt{n})
\end{align}
with $c_\mathrm{b}$ and $c_\mathrm{w}$ given in (\ref{randcon}). As in the normal limit cases in the examples in section \ref{ex1sec} we first need  asymptotic expressions for the variances.  We denote the variances of
 $B_n^\mathrm{b}$ and $B_n^\mathrm{w}$ with  $\hat{\sigma}^2_\mathrm{b}(n)$ and $\hat{\sigma}^2_\mathrm{w}(n)$.
These can be obtained from a result of  Matthews and Rosenberger \cite{maro97} for the number of draws of each color as follows:
\begin{lem}\label{async22ranb}
We have, as $n\to\infty$,
\begin{align}\label{async22ran}
\hat{\sigma}^2_\mathrm{b}(n) = f'_\mathrm{b} n + o(n), \quad \hat{\sigma}^2_\mathrm{w}(n) = f'_\mathrm{w} n + o(n),
\end{align}
with
\begin{align*}
f'_\mathrm{b}=f'_\mathrm{w}=\frac{\left(1-\alpha\right)\left(1-\beta\right)}{\left(1-\lambda\right)^{2}}
  \left(\frac{1}{1-2\lambda}-2\lambda\left(1+\lambda\right)\right)>0.
\end{align*}
\end{lem}
\begin{proof}
In \cite{maro97}, for the present urn model, the number $N_{n}$ of draws within the first $n$ draws in which a black
  ball is drawn is studied. Starting with one black ball it is established in \cite{maro97}, as
  $n\to\infty$, that
  \begin{align*}
    \E\left[N_{n}\right]&=\frac{1-\beta}{1-\lambda}\, n+o(n),\\
    \Var(N_{n}),
    &=\frac{\left(1-\alpha\right)\left(1-\beta\right)\left(3+2\lambda\right)}%
                  {\left(1-\lambda\right)^{2}\left(1-2\lambda\right)}\, n+o(n).
  \end{align*}
  As each black ball in the urn is either the first ball, or has been
  added after drawing a black ball and having success in tossing the
   corresponding coin, or after drawing a white ball and having no success in
  tossing the coin, we can directly link
  $N_{n}$ to $B^\mathrm{b}_n$:   Denoting the coin flips
  after drawing black balls by $(F^\mathrm{b}_j)_{1\le j\le N_{n}}$, the
  coin flips after drawing white balls by $(F^\mathrm{w}_j)_{1\le j\le
    (n-N_{n})}$ we have
  \begin{align*}
    B_n^{\mathrm{b}}
    &= 1
    +\sum_{j=1}^{N_{n}}F^\mathrm{b}_j
    +\sum_{j=1}^{n-N_{n}}\left(1-F^\mathrm{w}_j\right).
  \end{align*}
Using that all coin flips are independent we obtain from the  law of total variance by
  conditioning on $N_{n}$ that
  \begin{align*}
    \hat{\sigma}^{2}_{b}(n)
    &=\E\left[\Var\!\left(B_n^{\mathrm{b}}\bigm|N_{n}\right)\right]
      +\Var\!\left(\E\left[B_n^{\mathrm{b}}\bigm|N_{n}\right]\right)
    \\
    &=\frac{\left(1-\alpha\right)\left(1-\beta\right)}{\left(1-\lambda\right)^{2}}
      \left(\frac1{1-2\lambda}-2\lambda\left(1+\lambda\right)\right) n
      +o(n).
  \end{align*}
  When starting with one white ball, a similar argument gives the corresponding result.
\end{proof}
We use the normalizations $X_0:=Y_0:=0$ and, cf.~\eqref{scalb},
\begin{align}\label{normcase11b2}
X_n:=\frac{B_n^\mathrm{b}-\mu_\mathrm{b}(n)}{\hat{\sigma}_\mathrm{b}(n)},\quad
Y_n:=\frac{B_n^\mathrm{w}-\mu_\mathrm{w}(n)}{\hat{\sigma}_\mathrm{w}(n)}, \quad n\ge 1.
\end{align}
From the system (\ref{rder1})--(\ref{rder2})  we obtain for the scaled quantities $X_n$, $Y_n$ the system, for $n\ge 1$,
\begin{align*}
 X_n &\stackrel{d}{=} \frac{\hat{\sigma}_\mathrm{b}(I_n)}{\hat{\sigma}_\mathrm{b}(n)}X^{(1)}_{I_n}
+ F_{\alpha}\frac{\hat{\sigma}_\mathrm{b}(J_n)}{\hat{\sigma}_\mathrm{b}(n)} X^{(2)}_{J_n}
+ \left(1-F_{\alpha}\right) \frac{\hat{\sigma}_\mathrm{w}(J_n)}{\hat{\sigma}_\mathrm{b}(n)}  Y_{J_n} + e'_\mathrm{b}(n), \\
  Y_n &\stackrel{d}{=}   \frac{\hat{\sigma}_\mathrm{w}(I_n)}{\hat{\sigma}_\mathrm{w}(n)}  Y^{(1)}_{I_n} +F_{\beta}\frac{\hat{\sigma}_\mathrm{w}(J_n)}{\hat{\sigma}_\mathrm{w}(n)}   Y^{(2)}_{J_n} + \left(1-F_{\beta}\right)\frac{\hat{\sigma}_\mathrm{b}(J_n)}{\hat{\sigma}_\mathrm{w}(n)} X_{J_n}+ e'_\mathrm{w}(n),
\end{align*}
with conditions on independence and identical distributions analogously to (\ref{rder1})--(\ref{rder2}). We have  $\|e'_\mathrm{b}(n)\|_\infty, \|e'_\mathrm{w}(n)\|_\infty\to 0$ since the leading linear terms in the expansions (\ref{expmeanrannorm}) cancel out and the error terms $o(\sqrt{n})$ are asymptotically eliminated by the scaling of order $1/\sqrt{n}$.
 In view of  (\ref{coefcon}) this suggests for limits $X$ and $Y$ of $X_n$ and $Y_n$ respectively
\begin{align}\label{limsysran}
 X &\stackrel{d}{=} \sqrt{U}X^{(1)}
+ F_{\alpha}\sqrt{1-U} X^{(2)}
+ \left(1-F_{\alpha}\right) \sqrt{1-U} Y^{(1)}, \\
  Y &\stackrel{d}{=}   \sqrt{U}  Y^{(1)} +F_{\beta} \sqrt{1-U}  Y^{(2)} + \left(1-F_{\beta}\right)\sqrt{1-U} X^{(1)}, \label{limsysran1}
\end{align}
where $X^{(1)}$, $X^{(2)}$, $Y^{(1)}$, $Y^{(2)}$ and $U$ are independent and $X^{(1)}$, $X^{(2)}$ are distributed as $X$ and $Y^{(1)}$, $Y^{(2)}$ are distributed as $Y$.
To the map associated to the system (\ref{limsysran})--(\ref{limsysran1}) we can apply Theorem \ref{fp2}. The conditions (\ref{normcoe}) and (\ref{techcon}) are trivially satisfied. Hence $({\cal N}(0,1),{\cal N}(0,1))$ is the unique fixed-point of the associated map  in the space ${\cal M}^\R_3(0,1)\times {\cal M}^\R_3(0,1)$.
\begin{thm}\label{lim22nor2}
Consider the P\'olya urn with random replacement matrix (\ref{reprand2x2}) with $\alpha,\beta\in(0,1)$ and $\alpha+\beta < 3/2$ and the normalized numbers $X_n$ and $Y_n$  of black balls as in (\ref{normcase11b2}).
Then, as $n\to \infty$,
\begin{align*}
X_n \stackrel{d}{\longrightarrow} {\cal N}(0,1), \quad  Y_n \stackrel{d}{\longrightarrow} {\cal N}(0,1).
\end{align*}
\end{thm}
\begin{proof} Analogously to the proof of Theorem \ref{lim22nor}. \end{proof}
\noindent
{\bf Remark.} The case $\alpha+\beta=3/2$ differs in the error terms in
(\ref{expmeanrannorm}) which then become $\bo(\sqrt{n})$. Since the variances in (\ref{async22ran}) get additional logarithmic factors we still obtain the system (\ref{limsysran})--(\ref{limsysran1}) and our proof technique does still apply.

\subsection{Cyclic urns} \label{excyc}
We fix an integer $m\ge 2$ and consider an urn with balls of types
$1,\ldots,m$. After drawing a ball of type $j$ it is placed back to the urn together with a ball of type $j+1$ if $1\le j\le m-1$ and together with a ball of type $1$ if $j=m$. These urn models are called {\em cyclic urns}. Thus, the replacement matrix of a cyclic urn has the form
\begin{align}\label{repcyc}
R=\left[
  \begin{array}{ccccc}
    0 & 1 & &  & 0\\
       & 0  & 1  & & \\
       & & 0 & & \\
       &  &  & \ddots  &1 \\
     1 & & & &0
  \end{array}\right].
\end{align}
We denote by $R_n^{[j]}$ the number of type $1$ balls after $n$ draws when initially  one ball of type $j$ is  contained in the urn. Our recursive approach described above yields the system of recursive distributional equations
\begin{align}\label{cycscs}
R_n^{[1]} &\stackrel{d}{=} R_{I_n}^{[1]} + R_{J_n}^{[2]}, \\
R_n^{[2]} &\stackrel{d}{=} R_{I_n}^{[2]} + R_{J_n}^{[3]}, \nonumber\\
&\;\;\vdots  \nonumber \\
R_n^{[m]}&\stackrel{d}{=} R_{I_n}^{[m]} + R_{J_n}^{[1]},\nonumber
\end{align}
where, on the right hand sides, $I_n$ and $R_k^{[j]}$ for $j=1,\ldots,m$, $k=0,\ldots,n-1$ are independent, $I_n$ uniformly distributed on $\{0,\ldots,n-1\}$ and $J_n=n-1-I_n$.

We denote the imaginary unit by $\mathrm{i}$ and use the  primitive roots of unity
\begin{align}\label{rootofunity}
  \omega:= \omega_m:=
  \exp\left(\frac{2\pi \mathrm{i}}{m}\right)=: \lambda + \mathrm{i}\mu
\end{align}
with $\lambda,\mu\in\R$. Note that for $2\le m\le 6$ we have $\lambda\le 1/2$, while for $m\ge 7$ we have $\lambda>1/2$. Asymptotic expressions for the mean of the $R_{n}^{[j]}$ can be found (together with further analysis) in \cite{Ja83, Ja04, Pou05}. To keep this section self-contained we give an exact formula for later use:
\begin{lem}\label{lemewcyc}
  Let $R_{n}^{[j]}$ be the number of balls of color 1 after $n$ draws in a cyclic urn with $m\ge 2$  colors, starting with one ball of color $j$. Then, with $\omega=\omega_m$ as in (\ref{rootofunity}) we have
  \begin{align}\label{cycew}
    \E\!\left[R_{n}^{[j]}\right]=\frac{n+1}{m}+\frac1m\sum_{k\in \{1,\ldots,m-1\}\setminus \{m/2\}}
    \frac{\Gamma(n+1+\omega^{k})}{\Gamma(n+1)\,\Gamma(\omega^{k}+1)}\omega^{k(j-1)}.
  \end{align}
In particular, we have $\E[R_{n}^{[j]}]=\frac{1}{m}n + \bo(1)$ for $m=2,3,4$ and, for   $m>4$, as $n\to\infty$,
  \begin{align}\label{cycew2}
    \E\left[R_n^{[j]}\right]=\frac{1}{m}n+\Re(\kappa_j n^{\mathrm{i}\mu}) n^\lambda + o(n^\lambda),
    \qquad \kappa_j:=\frac{2\omega^{j-1}}{m\Gamma(\omega+1)}.
  \end{align}
\end{lem}
\begin{proof}
  Using the system (\ref{cycscs}), we obtain by conditioning on $I_{n}$, for any $1\le j\le m$,
  \begin{align*}
    \E\left[R_n^{[j]}\right]
    &=\frac1n\sum_{i=0}^{n-1}\E\left[R_i^{[j]}\right]+\frac1n\sum_{i=0}^{n-1}\E\left[R_i^{[j+1]}\right]
    \\
    &=\frac1n\left(\E\left[R_{n-1}^{[j]}\right]+\E\left[R_{n-1}^{[j+1]}\right]\right)
      +\frac{n-1}n \E\left[R_{n-1}^{[j]}\right]
    \\
    &=\E\left[R_{n-1}^{[j]}\right]+\frac1n\E\left[R_{n-1}^{[j+1]}\right],
  \end{align*}
  where we set $R_{i}^{[m+1]}:=R_{i}^{[1]}$ for any $1\le i\le n$.
  With the column vector $R_{n}:=(R_{n}^{[1]},\ldots,R_{n}^{[m]})$,
  the replacement matrix $R$ in (\ref{repcyc}) and the identity matrix
  $\mathrm{Id}_{m}$ this is rewritten as
  \begin{align*}
    \E\left[R_{n}\right]
    &=\left(\mathrm{Id}_{m}+\frac1n R\right)\E\left[R_{n-1}\right]
    =\prod_{k=1}^{n}\left(\mathrm{Id}_{m}+\frac1k R\right)\,\E\!\left[R_{0}\right].
  \end{align*}
  The eigenvalues of the replacement matrix are all $m$-th roots of
  unity $
  \omega^{k}$, $k=1,\dots,m$, and a possible eigenbasis is
  $v_{k}\coleq\frac1m(\omega^{0},\omega^{k},\dots,\omega^{\left(m-1\right)k})^{t}$,
  $k=1,\dots,m$. Decomposing the mapping induced by $R$ into the
  projections $\pi_{v_{k}}$ onto the respective eigenspaces
 we obtain
  \begin{align*}
    \prod_{\ell=1}^{n}\left(\mathrm{Id}_{m}+\frac1\ell R\right)
    &=
    \sum_{k=1}^{m}\prod_{\ell=1}^{n}\left(1+\frac1\ell\omega^{k}\right)\pi_{v_{k}}
    \\
    &=\left(n+1\right)\, \pi_{v_{m}}
     +\!\!\sum_{k\in\{1,\dots,m-1\}\setminus\{m/2\}}
       \frac{\Gamma\!\left(n+1+\omega^{k}\right)}%
       {\Gamma\!\left(\omega^{k}+1\right)\Gamma\!\left(n+1\right)}
       \,\pi_{v_{k}}.
  \end{align*}
  Moreover, $\pi_{v_{k}}(\E[R_{0}])=v_{k}$ and
  $v_{m}=\frac1m(1,\dots,1)$, hence the $j$-th component of the latter
  display implies (\ref{cycew}). The asymptotic expansion in
  (\ref{cycew2}) is now directly read off: Note that the roots of
  unity come in conjugate pairs $\omega^{m-k}=\overline{\omega}^{k}$.
  If $m$ is even, $\omega^{m/2}=\overline{\omega}^{m/2}=-1$, otherwise
  only $\omega^{m}=1$ is real. Combining pairs of summands for such
  conjugate pairs and using $\Gamma(\overline{z})=\overline{\Gamma(z)}$, we
  obtain the terms
  \begin{align*}
    \frac{\Gamma\!\left(n+1+\omega^{k}\right)\,\omega^{\left(j-1\right)k}}%
    {\Gamma\!\left(n+1\right)\,\Gamma\!\left(\omega^{k}+1\right)}
    +\frac{\Gamma\!\left(n+1+\overline{\omega}^{k}\right)
      \overline{\omega}^{\left(j-1\right)k}}%
    {\Gamma\!\left(n+1\right)\,\Gamma\!\left(\overline{\omega}^{k}+1\right)}
    =2\,\Re\!\left(
       \frac{\omega^{\left(j-1\right)k}\, \Gamma(n+1+\omega^{k})}%
         {\Gamma\!\left(\omega^{k}+1\right) \,\Gamma\!\left(n+1\right)}
       \right).
  \end{align*}
  By Stirling approximation the asymptotic growth order of the latter
  term is $\Re(n^{\omega^{k}})$, hence the dominant asymptotic term is
  for the conjugate pair with largest real part, $\omega$ and
  $\omega^{m-1}$. This implies (\ref{cycew2}) for $m>4$. For $m=3,4$
  the periodic term is $o(1)$ respectively $\bo(1)$, for $m=2$ there
  is no periodic fluctuation.
\end{proof}
We do not discuss limit laws for the cases $2\le m\le 6$ in detail.
They lead to asymptotic normality as has been shown with different
proofs in Janson \cite{Ja83} and Janson \cite[Example 7.9]{Ja04}.
These cases can be covered by our approach similarly to the normal
cases in sections \ref{ex1sec} and \ref{kersting}. For $2\le m\le 6$,
the system of limit equations is
\begin{align*}
X^{[1]} &\stackrel{d}{=} \sqrt{U} X^{[1]} +  \sqrt{1-U} X^{[2]}, \\
X^{[2]} &\stackrel{d}{=} \sqrt{U} X^{[2]} + \sqrt{1-U}  X^{[3]}, \\
&\;\;\vdots \\
X^{[m]}&\stackrel{d}{=} \sqrt{U} X^{[m]} + \sqrt{1-U} X^{[1]},
\end{align*}
and Theorem \ref{fp2} applies.

We now assume $m\ge 7$.
In particular, we have the  asymptotic expansion (\ref{cycew2}) of the mean of the $R_n^{[j]}$ with $\lambda>1/2$. We define the normalizations
\begin{align}\label{defXn2}
X_n^{[j]}:= \frac{R_n^{[j]}-\frac{1}{m}n}{n^\lambda}.
\end{align}
Hence, we obtain for the $X_n^{[j]}$ the system
\begin{align*}
X_n^{[1]} &\stackrel{d}{=} \left(\frac{I_n}{n}\right)^\lambda X_{I_n}^{[1]} +  \left(\frac{J_n}{n}\right)^\lambda X_{J_n}^{[2]}-\frac{1}{mn^\lambda}, \\
X_n^{[2]} &\stackrel{d}{=} \left(\frac{I_n}{n}\right)^\lambda X_{I_n}^{[2]} + \left(\frac{J_n}{n}\right)^\lambda X_{J_n}^{[3]}-\frac{1}{mn^\lambda}, \\
&\;\;\vdots \\
X_n^{[m]}&\stackrel{d}{=} \left(\frac{I_n}{n}\right)^\lambda X_{I_n}^{[m]} + \left(\frac{J_n}{n}\right)^\lambda X_{J_n}^{[1]}-\frac{1}{mn^\lambda},
\end{align*}
where, on the right hand sides, $I_n$ and $X_k^{[j]}$ for $j=1,\ldots,m$, $k=0,\ldots,n-1$ are independent.
To describe the asymptotic periodic behavior of the distributions of the $X_n^{[j]}$, we use the following related  system of limit equations:
\begin{align*}
X^{[1]} &\stackrel{d}{=} U^\omega X^{[1]} +  (1-U)^\omega X^{[2]}, \\
X^{[2]} &\stackrel{d}{=} U^\omega X^{[2]} + (1-U)^\omega  X^{[3]}, \\
&\;\;\vdots \\
X^{[m]}&\stackrel{d}{=} U^\omega X^{[m]} + (1-U)^\omega X^{[1]}.
\end{align*}
Since $\omega$ is complex nonreal this now has to be considered as a system to solve for distributions ${\cal L}(X^{[1]}),\ldots, {\cal L}(X^{[m]})$ on the complex plane $\C$.
The corresponding map $\bar{T}$ is a special case of $T'$ in (\ref{defT'}):
\begin{align}
\bar{T}: {\cal M}^{\C,\times m}
&\to {\cal M}^{\C,\times m} \nonumber\\
(\mu_1,\ldots,\mu_{m}) &\mapsto (\bar{T}_1(\mu_1,\ldots,\mu_{m}),\ldots, \bar{T}_{m}(\mu_1,\ldots,\mu_{m})) \nonumber\\
\bar{T}_j(\mu_1,\ldots,\mu_{m}) &:= {\cal L}\left(U^\omega V^{[j]} + (1-U)^\omega  V^{[j+1]}\right)
\label{defT}
\end{align}
for $j=1,\ldots,m$ with $U, V^{[1]}, \ldots,V^{[m+1]}$ independent, $U$ uniformly distributed on $[0,1]$ and ${\cal L}(V^{[j]})=\mu_j$ for $j=1,\ldots,m$ and ${\cal L}(V^{[ m+1]})=\mu_1$.
\begin{lem}\label{lem_fixed}
Let $m\ge 7$.
The restriction of   $\bar{T}$  to ${\cal M}^\C_2(\kappa_1)\times \cdots \times {\cal M}^\C_2(\kappa_{m})$   has a unique fixed-point.
\end{lem}
\begin{proof}
We verify the conditions of Theorem \ref{fp3}: First note that condition (\ref{cccent}) for our $\bar{T}$ in (\ref{defT}) is
\begin{align}\label{conappl}
\E\left[U^\omega\right] \kappa_j + \E\left[(1-U)^\omega\right]  \kappa_{j+1}=\kappa_j, \quad j=1,\ldots,m,
\end{align}
with $\kappa_{m+1}:=\kappa_1$. Since $\E[U^\omega] =\E[(1-U)^\omega]=(1+\omega)^{-1}$ and
 $\kappa_{j+1}= \omega \kappa_j$ we find that (\ref{conappl}) is satisfied. Condition (\ref{cccont}) for our $\bar{T}$  is
 \begin{align*}
 \E\left[ \left|U^{2\omega}\right|\right]+  \E\left[ \left|(1-U)^{2\omega}\right|\right]<1.
\end{align*}
Since $m\ge 7$, we have $\lambda > 1/2$, thus $\E[ |U^{2\omega}|]+  \E[ |(1-U)^{2\omega}|] =2/(1+2\lambda)<1$.
Hence Theorem \ref{fp3} applies and implies the assertion.
\end{proof}
The fixed-point  in Lemma \ref{lem_fixed} has a particularly simple structure as follows. Note that a  description related to (\ref{fpesimpli}) was given in Remark 2.3 in Janson \cite{Ja06}.
\begin{lem}\label{lem_fixed_3}
Let $m\ge 7$ and $({\cal L}(\Lambda^{[1]}),\ldots,{\cal L}(\Lambda^{[m]}))$ be the unique fixed-point  in Lemma \ref{lem_fixed}. Furthermore let ${\cal L}(\Lambda)$ be the (unique) fixed-point of
\begin{align}\label{fpesimpli}
X\stackrel{d}{=}U^\omega X + \omega (1-U)^\omega X' \quad \mbox{ in } \quad {\cal M}^\C_2\left(\frac{2}{m\Gamma(\omega+1)}\right),
\end{align}
where $X$, $X'$ and $U$ are independent, $U$ is uniformly distributed on $[0,1]$  and $X$ and $X'$ have identical distributions. Then we have
\begin{align*}
 \Lambda^{[j]} \stackrel{d}{=} \omega^{j-1} \Lambda, \quad j=1,\ldots,m.
\end{align*}
\end{lem}
\begin{proof}We abbreviate $\gamma:=2/(m\Gamma(\omega+1))$.
For $X$, $X'$ and $U$ independent, $U$ uniformly on $[0,1]$ distributed, $X$ and $X'$ identically distributed with $\E X=\gamma$ we have
\begin{align*}
 \E\left[U^\omega X + \omega (1-U)^\omega X'\right]=\frac{1}{1+\omega}(\gamma +\omega\gamma)=\gamma,
\end{align*}
hence the map of probability measures on $\C$ associated to (\ref{fpesimpli}) maps ${\cal M}^\C_2(\gamma)$ into itself. The argument of the proof of Theorem \ref{fp3} implies that
this map is a contraction on $({\cal M}^\C_2(\gamma),\ell_2)$. Hence it has a unique fixed point ${\cal L}(\Lambda)$. We have
\begin{align*}
 ({\cal L}(\Lambda), {\cal L}(\omega\Lambda),\ldots,{\cal L}(\omega^{m-1}\Lambda))\in {\cal M}^\C_2(\kappa_1)\times \cdots \times {\cal M}^\C_2(\kappa_{m})
\end{align*}
and, by plugging into (\ref{defT}), we find that this vector is a fixed-point of  $\bar{T}$. Since, by Lemma \ref{lem_fixed}, there is only one fixed-point of $\bar{T}$ in ${\cal M}^\C_2(\kappa_1)\times \cdots \times {\cal M}^\C_2(\kappa_{m})$ the assertion follows.
\end{proof}
The asymptotic periodic behavior in the following theorem has already been shown almost surely by martingale methods in \cite[Section 4.2]{Pou05}, see also \cite[Theorem 3.24]{Ja04}. Our contraction approach adds the characterization of ${\cal L}(\Lambda)$ as the fixed-point in (\ref{fpesimpli}). The proof is based on the complex version of the $\ell_2$ metric and resembles ideas from Fill and Kapur \cite{FiKa04}, see also \cite[Theorem 5.3]{JaNe08}.
\begin{thm}
Let $m\ge 7$ and $X_n^{[j]}$ as in (\ref{defXn2})
 and  ${\cal L}(\Lambda)$ the unique fixed-point in Lemma \ref{lem_fixed_3}.
Then, for all $j=1,\ldots,m$,  we have
\begin{align} \label{limperic}
\ell_2\left( X_n^{[j]}, \Re\left(e^{\mathrm{i}\left(\mu \ln(n) + 2\pi \frac{j-1}{m}\right)}\Lambda\right) \right) \to 0 \quad (n\to \infty).
\end{align}
\end{thm}
\begin{proof}
Let $\Lambda^{[1]},\ldots,\Lambda^{[m]}$ be independent random variables such that $({\cal L}(\Lambda^{[1]}),\ldots,{\cal L}(\Lambda^{[m]}))$ is the unique fixed-point as in Lemma \ref{lem_fixed}. Set $\Lambda^{[m+1]}:=\Lambda^{[1]}$. Note that for the random variable within the real part in (\ref{limperic}) with Lemma \ref{lem_fixed_3} we have
\begin{align*}
e^{\mathrm{i}\left(\mu \ln(n) + 2\pi \frac{j-1}{m}\right)}\Lambda = n^{\mathrm{\mathrm{i}}\mu} \omega^{j-1} \Lambda \stackrel{d}{=}n^{\mathrm{i}\mu}\Lambda^{[j]}.
\end{align*}
The fixed-point property of the $\Lambda^{[j]}$  implies
\begin{align*}
\Re\left(n^{\mathrm{i}\mu} \Lambda^{[j]}\right) \stackrel{d}{=} \Re\left(n^{\mathrm{i}\mu} U^\omega \Lambda^{[j]}\right) + \Re\left(n^{\mathrm{i}\mu} (1-U)^\omega \Lambda^{[j+1]}\right)
\end{align*}
for all $j=1,\ldots,m$ and $n\ge 0$. We denote
\begin{align*}
\Delta_j(n):= \ell_2\left(X_n^{[j]}, \Re\left(n^{\mathrm{i}\mu} \Lambda^{[j]}\right)\right)
\end{align*}
and set $\Delta_{m+1}(n):=\Delta_1(n)$.
Now, we assume that the $X_n^{[j]}$, $\Lambda^{[j]}$, $n\ge 1$, $1\le j\le m$, $I_n$, $U$ appearing in (\ref{defXn2}) and (\ref{defT}) are defined on one probability space such that
$(X_n^{[j]}, \Re(n^{\mathrm{i}\mu} \Lambda^{[j]}))$ are optimal $\ell_2$-couplings for all $n\ge 0$ and all  $1\le j\le m$ and such that $I_n=\lfloor nU \rfloor$. Then we have
\begin{align} \label{basic_est}
\lefteqn{\Delta_j(n)} \nonumber\\
&= \ell_2\left( \left(\frac{I_n}{n}\right)^\lambda X_{I_n}^{[j]} + \left(\frac{J_n}{n}\right)^\lambda X_{J_n}^{[j+1]}-\frac{1}{mn^\lambda},  \Re\left(n^{\mathrm{i}\mu} U^\omega \Lambda^{[j]}\right) + \Re\left(n^{\mathrm{i}\mu} (1-U)^\omega \Lambda^{[j+1]}\right) \right)\nonumber\\
&\le \left\|\left\{ \left(\frac{I_n}{n}\right)^\lambda X_{I_n}^{[j]} - \Re\left(\frac{I_n^\omega}{n^\lambda} \Lambda^{[j]}\right)\right\} +\left\{ \left(\frac{J_n}{n}\right)^\lambda X_{J_n}^{[j+1]} - \Re\left(\frac{J_n^\omega}{n^\lambda} \Lambda^{[j+1]}\right)\right\} \right\|_2 \nonumber\\
& \quad~ + \left\|\Re\left(\frac{I_n^\omega}{n^\lambda} \Lambda^{[j]}\right) - \Re\left(n^{\mathrm{i}\mu} U^\omega \Lambda^{[j]}\right) \right\|_2 +  \left\|\Re\left(\frac{J_n^\omega}{n^\lambda} \Lambda^{[j+1]}\right) - \Re\left(n^{\mathrm{i}\mu} U^\omega \Lambda^{[j+1]}\right) \right\|_2 + \frac{1}{mn^\lambda}\nonumber \\
&=: S_1+S_2+S_3+ \frac{1}{mn^\lambda}.
\end{align}
First note that the summands $S_2$ and $S_3$ tend to zero: We have $(I_n/n)^\omega \to U^\omega$ almost surely by $I_n=\lfloor nU \rfloor$. Since $\Lambda^{[j]}$ and $\Lambda^{[j+1]}$ have finite second moments we can apply dominated convergence to obtain
$S_2,S_3 \to 0$ as $n\to\infty$.

For the estimate of the first summand $S_1$ we abbreviate
\begin{align*}
W_{n}^{[j]}:= \left(\frac{I_n}{n}\right)^\lambda X_{I_n}^{[j]} - \Re\left(\frac{I_n^\omega}{n^\lambda} \Lambda^{[j]}\right),\qquad
W_{n}^{[j+1]}:= \left(\frac{J_n}{n}\right)^\lambda X_{J_n}^{[j+1]} - \Re\left(\frac{J_n^\omega}{n^\lambda} \Lambda^{[j+1]}\right).
\end{align*}
Then we have
 \begin{align} \label{binom_form}
S_1^2= \E\left[(W_{n}^{[j]})^2\right] + \E\left[(W_{n}^{[j+1]})^2\right] + 2\E\left[W_{n}^{[j]}W_{n}^{[j+1]}\right]
\end{align}
Conditioning on $I_n$ and using that $(X_k^{[j]}, \Re(k^{\mathrm{i}\mu} \Lambda^{[j]}))$ are optimal $\ell_2$-couplings we obtain
\begin{align*}
\E\left[(W_{n}^{[j]})^2\right]&=\sum_{k=0}^{n-1} \frac{1}{n} \E\left[\left\{ \left(\frac{k}{n}\right)^\lambda X_{k}^{[j]} - \Re\left(\frac{k^\lambda k^{\mathrm{i}\mu}}{n^\lambda} \Lambda^{[j]}\right) \right\}^2 \right]\\
&=\sum_{k=0}^{n-1} \frac{1}{n}\left(\frac{k}{n}\right)^{2\lambda}\E\left[\left\{ X_{k}^{[j]} - \Re\left( k^{\mathrm{i}\mu} \Lambda^{[j]}\right) \right\}^2 \right]\\
&=\sum_{k=0}^{n-1} \frac{1}{n}\left(\frac{k}{n}\right)^{2\lambda} \Delta_j^2(k)\\
&= \E\left[ \left(\frac{I_n}{n}\right)^{2\lambda}\Delta_j^2(I_n)    \right].
\end{align*}
Analogously, we have
\begin{align*}
\E[(W_{n}^{[j+1]})^2]
= \E\left[ \left(\frac{J_n}{n}\right)^{2\lambda}\Delta_{j+1}^2(J_n)  \right].
\end{align*}
To bound the mixed term in (\ref{binom_form}) note that by the expansion (\ref{cycew2}) and the normalization (\ref{defXn2}) we have  $\E[X_n^{[j]}]=\Re(\kappa_j n^{\mathrm{i}\mu}) + r_j(n)$ with $r_j(n)\to 0$ as $n\to \infty$ for all $j=1,\ldots,m$.   In particular, we have $\|r_j\|_\infty<\infty$. Together with  $\E[\Lambda^{[j]}]=\kappa_j$ this implies $\E[W_{n}^{[j]}]=\E[(I_n/n)^\lambda r_j(I_n)]$ and
\begin{align} \label{mixed_term}
\E\left[W_{n}^{[j]}W_{n}^{[j+1]}\right]= \E\left[ \left(\frac{I_n}{n}\frac{J_n}{n}\right)^\lambda r_j(I_n) r_{j+1}(J_n)  \right].
\end{align}
To show that the latter term tends to zero let $\varepsilon>0$. Then there exists $k_0\in\Nat$ such that $r_j(k)<\varepsilon$, $r_{j+1}(k)<\varepsilon$ for all $k\ge k_0$. For all $n> 2k_0$ we obtain, by considering the event $\{k_0\le I_n\le n-1-k_0\}$ and its complement,
\begin{align*}
\E\left[W_{n}^{[j]}W_{n}^{[j+1]}\right] \le \frac{2k_0}{n} \|r_j\|_\infty\|r_{j+1}\|_\infty + \varepsilon^2.
\end{align*}
Hence, we obtain that the mixed term (\ref{mixed_term}) tends to zero.

Altogether, we obtain from (\ref{basic_est}) as $n\to\infty$ that
\begin{align*}
\Delta_j(n) &\le \left\{\E\left[ \left(\frac{I_n}{n}\right)^{2\lambda}\Delta_j^2(I_n) \right] +  \E\left[ \left(\frac{J_n}{n}\right)^{2\lambda}\Delta_{j+1}^2(J_n)  \right]+o(1)\right\}^{1/2} + o(1)\\
&\le  \left\{2\E\left[ \left(\frac{I_n}{n}\right)^{2\lambda}\Delta^2(I_n) \right]+o(1)\right\}^{1/2} + o(1),
\end{align*}
for all $j=1,\ldots,m$,
where
\begin{align*}
\Delta(n):=\max_{1\le j\le m} \Delta_j(n).
\end{align*}
Hence, we have
\begin{align} \label{rec_finish}
\Delta(n) \le  \left\{2\E\left[ \left(\frac{I_n}{n}\right)^{2\lambda}\Delta^2(I_n) \right]+o(1)\right\}^{1/2} + o(1).
\end{align}
Now, we obtain $\Delta(n)\to 0$ as in the proof of Theorem \ref{22nnc}: First from (\ref{rec_finish}) we obtain with $I_n/n\to U$ almost surely that
\begin{align*}
\Delta(n) &\le  \left\{2\E\left[ \left(\frac{I_n}{n}\right)^{2\lambda}\right]\max_{0\le k\le n-1} \Delta^2(k)+o(1)\right\}^{1/2} + o(1)\\
&\le  \left\{\left(\frac{2}{1+2\lambda}+o(1)\right)\max_{0\le k\le n-1} \Delta^2(k)+o(1)\right\}^{1/2} + o(1).
\end{align*}
Since $\lambda>1/2$ this implies that the sequence $(\Delta(n))_{n\ge 0}$ is bounded. We denote $\eta:=\sup_{n\ge 0} \Delta(n)$ and $\xi:=\limsup_{n\to\infty} \Delta(n)$. For any $\varepsilon> 0$ there exists an $n_0 \ge 0$ such that $\Delta(n) \le \xi+\varepsilon$  for all
$n \ge n_0$. Hence, from (\ref{rec_finish}) we obtain
\begin{align*}
\Delta(n) \le  \left\{2\E\left[{\bf 1}_{\{I_n<n_0\}} \left(\frac{I_n}{n}\right)^{2\lambda} \right]\eta^2+ 2\E\left[{\bf 1}_{\{I_n\ge n_0\}} \left(\frac{I_n}{n}\right)^{2\lambda} \right](\xi+\varepsilon)^2  +   o(1)\right\}^{1/2} + o(1).
\end{align*}
With $n\to\infty$ this implies
\begin{align*}
\xi \le  \sqrt{\frac{2}{1+2\lambda}}(\xi+\varepsilon).
\end{align*}
Since $\sqrt{2/(1+2\lambda)}<1$ and $\varepsilon>0$ is arbitrary this implies $\xi=0$.
\end{proof}

\section{Remarks on the use of the contraction method}\label{sec7}
 A novel technical aspect of this paper is that we extend the use of the contraction method to systems of recursive distributional equations. Alternatively, one may be tempted to couple the random variables
 $B^\mathrm{b}_n$ and $B^\mathrm{w}_n$ in (\ref{rde1}) and (\ref{rde2}) on one probability space, setup  a recurrence for their vector $(B^\mathrm{b}_n,B^\mathrm{w}_n)$ and try to apply general transfer theorems from the contraction method for multivariate recurrences, such as Theorem 4.1 in Neininger \cite{Ne01} or Theorem 4.1 in Neininger and R\"uschendorf \cite{NeRu04}. For some particular instances (replacement schemes) of the P\'olya urn this is in fact possible. However, when attempting to come up with a limit theory of the generality of the present paper such a multivariate approach hits two nags that seem difficult to overcome. In this section we highlight these problems at one of the examples discussed above and explain why we consider such a multivariate approach disadvantageous  in the context of P\'olya urns.

 We consider the example from section \ref{kersting} with the random replacement matrix in (\ref{reprand2x2}) and denote the bivariate random variable by $B_n:=(B^\mathrm{b}_n,B^\mathrm{w}_n)$ with $B^\mathrm{b}_n$ and $B^\mathrm{w}_n$ as in (\ref{rder1}) and (\ref{rder2}) respectively.
  Note that in the  discussion of section \ref{kersting} the random variables $B^\mathrm{b}_n$ and $B^\mathrm{w}_n$ did
not need to be defined on a common probability space. Hence, first of
all, only the marginals of $B_n$ are determined by
the urn process and we have the choice of a joint distribution for
$B_n$ respecting these marginals. We could keep the
components independent or choose appropriate couplings. We choose  a form that implies a recurrence of the form typically considered in  general limit theorems from the contraction method: The coupling is defined recursively by
$B_0=(1,0)$ and, for $n\geq 1$,
\begin{align}\label{bivdef}
B_n:\stackrel{d}{=} B_{I_n} +
\left[\!\!
\begin{array}{cc} F_\alpha & 1- F_\alpha \\
1-F_\beta & F_\beta \end{array}\!\!\right]
B'_{J_n},
\end{align}
where $(B_n)_{0\leq k<n}$, $(B'_n)_{0\leq k<n}$, $(F_\alpha,F_\beta)$, and $I_n$ are
independent and $B_k$ and $B'_k$ identically distributed  for all $0\le k<n$. As in section \ref{kersting}, $I_n$ is uniformly
distributed on $\{0,\dots,n-1\}$ and $J_n:=n-1-I_n$, while $F_\alpha$ and $F_\beta$ are Bernoulli random
variables being 1 with probabilities $\alpha$ and $\beta$
respectively, otherwise 0. Note that for any joint distribution of
$(F_\alpha,F_\beta)$, definition (\ref{bivdef}) leads to
a sequence $(B_n)_{n\geq1}$ with correct marginals of
$B^\mathrm{b}_n$ and $B^\mathrm{w}_n$.  A beneficial joint distribution of
$(F_\alpha,F_\beta)$
 will be chosen below.

We consider the cases where $\alpha+\beta-1<1/2$. Since these lead to normal limits one may try to apply Theorem 4.1 in \cite{NeRu04} where $2<s\le 3$ there is the index of the Zolotarev metric $\zeta_s$ on which that Theorem is based. The best possible contraction condition, cf.~equation (25) in \cite{NeRu04} is obtained with $s=3$ which we fix subsequently. Now, for the application of Theorem 4.1 in \cite{NeRu04} we need an asymptotic expansion of the covariance matrix of $B_n$. In view of Lemma \ref{async22ranb} we assume that  for all $i,j=1,2$ we have
\begin{align}\label{napp2}
(\Cov(B_n))_{ij}= f_{ij} n + o(n), \quad (n\to \infty)
\end{align}
such that $(f_{ij})_{ij}$ is a symmetric, positive definite $2\times 2$ matrix. Hence there exists an $n_1\ge 1$ such that $\Cov(B_n)$ is positive definite for all $n\ge n_1$. For the normalized random sequence
\begin{align*}
X_n:= (\Cov(B_n))^{-1/2}(B_n-\E[B_n]), \quad n\ge n_1,
\end{align*}
we obtain the limit equation
\begin{align*}
X\stackrel{d}{=} \sqrt{U}X+ \sqrt{1-U}\left[\!\!
\begin{array}{cc} F_\alpha & 1- F_\alpha \\
1-F_\beta & F_\beta \end{array}\!\!\right] X',
\end{align*}
where $X,X',U,(F_\alpha,F_\beta)$ are independent, $X$ and $X'$ are identically distributed and $U$ is uniformly distributed on $[0,1]$. Now the application of Theorem 4.1 in \cite{NeRu04} requires condition (25) there to be satisfied which in our example writes
\begin{align}\label{napp1}
\E\left[U^{3/2}\right]+\E\left[(1-U)^{3/2}\right]\E\left[ \left\| \left[\!\!
\begin{array}{cc} F_\alpha & 1- F_\alpha \\
1-F_\beta & F_\beta \end{array}\!\!\right] \right\|_{\mathrm{op}}^{3}\right]<1,
\end{align}
where $\|\,\cdot\,\|_\mathrm{op}$ denotes the operator norm of the matrix.
Here, the joint distribution of $(F_\alpha,F_\beta)$ can be chosen to minimize the left hand side of the latter inequality as follows: For $V$ uniformly distributed and independent of $U$ we set $F_\alpha = {\bf 1}_{\{V\le \alpha\}}$ and $F_\beta= {\bf 1}_{\{V\le \beta\}}$. With this choice of the joint distribution of $(F_\alpha,F_\beta)$ condition (\ref{napp1}) turns into
\begin{align*}
\frac{2}{5}\left( 2+ |\alpha-\beta|(2^{3/2}-1)\right)<1.
\end{align*}
We see that this condition is not satisfied in the whole range $\alpha+\beta-1<1/2$. Hence, in the best possible setup that we could find Theorem 4.1 in \cite{NeRu04} does not  yield results of the strength of Theorem \ref{lim22nor2}.

A second drawback of the use of multivariate recurrences is that we needed the assumption of the expansion (\ref{napp2}), which technically   is required in order to verify condition (24) in \cite{NeRu04}. Hence, after coupling $B^\mathrm{b}_n$ and $B^\mathrm{w}_n$ on one probability space such that we may satisfy (\ref{napp1}) we have to derive  asymptotic expressions for the covariance $\Cov(B^\mathrm{b}_n,B^\mathrm{w}_n)$ and to identify the leading constant in these asymptotics. Note, that this covariance is meaningless for the P\'olya urn and only emerges by artificially coupling $B^\mathrm{b}_n$ and $B^\mathrm{w}_n$. This covariance does not appear in the approach we propose in section \ref{conexp}, which makes its application much simpler compared to a multivariate formulation.

A reason why our approach of analysing systems of recurrences is  more powerful than the use of multivariate recurrences is found when comparing the spaces of probability measures where one tries to apply contraction arguments on: In section \ref{secspaces} we introduce the space
$({\cal M}^\R_s)^{\times d}$ in (\ref{defspace}) and work on subspaces where first or first and second moments of the
  probability measures are fixed. The corresponding space
  in a multivariate formulation and in Theorem 4.1
  in \cite{NeRu04} is the space ${\cal M}_s(\R^d)$ of
  all probability measures on $\R^d$ with finite
  absolute $s$-th moment. Clearly $({\cal M}^\R_s)^{\times d}$ is much smaller than ${\cal M}_s(\R^d)$, e.g.~the
  first space can be embedded into the second by
  forming product measures. This makes it plausible that
  it is much easier to find contracting maps as developed
  in section \ref{assofix} on $({\cal M}^\R_s)^{\times d}$
  than on ${\cal M}_s(\R^d)$ and we feel that this causes the problems mentioned above with a multivariate formulation.

  In the dissertation Knape  \cite[Chapter 5]{kn13} more details on our use of the contraction method and an alternative multivariate formulation are given. There, also  improved versions of Theorem 4.1 in \cite{NeRu04} are derived by a change of the underlying probability metric which lead to better conditions compared to (\ref{napp1}). However, the necessity to derive artificial covariances in a multivariate approach as discussed above could not be surmounted in \cite{kn13}. Similar
  advantages of the use of systems of recurrences over multivariate formulations were  noted in Leckey et.~al.~\cite[Section 7]{lenesz13}.



\end{document}